\documentclass[12pt, reqno]{amsart}
\usepackage{amsmath, amstext, amsbsy, amssymb}

\newtheorem{theorem}{Theorem}[section]
\newtheorem{lemma}[theorem]{Lemma}
\newtheorem{proposition}[theorem]{Proposition}

\theoremstyle{definition}
\newtheorem{definition}[theorem]{Definition}

\theoremstyle{remark}
\newtheorem{remark}[theorem]{Remark}
\theoremstyle{conjecture}
\newtheorem{conjecture}[theorem]{Conjecture}

\numberwithin{equation}{section}

\newcommand{\cp}{\mathcal P}
\newcommand{\ch}{{\rm ch} }
\newcommand{\C}{ \mathbb C }
\newcommand{\Coe}{ {\rm Coeff} }

\newcommand{\End}{{\rm End}}
\newcommand{\fa}{ \mathfrak a }
\newcommand{\fd}{ \mathfrak d }
\newcommand{\fG}{ \mathfrak G }
\newcommand{\fL}{ \mathfrak L }
\newcommand{\fock}{{\mathbb H}_X}

\newcommand{\Hn}{H^*(\Xn)}

\newcommand{\la}{{\lambda}}
\newcommand{\lambsq}{s(\lambda)}

\newcommand{\Ln}{L^{[n]}}

\newcommand{\Tr}{ {\rm Tr} }
\newcommand{\vac}{|0\rangle}
\newcommand{\w}{\tilde}

\newcommand{\Wcp}{\widetilde{\mathcal P}}
\newcommand{\Xn}{ {X^{[n]}}}
\newcommand{\Z}{ \mathbb Z }

\def\beq{\begin{equation}}
\def\eeq{\end{equation}}

\numberwithin{equation}{section}

\begin{document}

\title[Okounkov's conjecture]
      {On Okounkov's conjecture connecting Hilbert schemes of points 
          and multiple $q$-zeta values}

\author[Zhenbo Qin]{Zhenbo Qin$^1$}
\address{Department of Mathematics, University of Missouri, Columbia, MO
65211, USA} \email{zq@math.missouri.edu}
\thanks{${}^1$Partially supported by a grant from the Simons Foundation}

\author[Fei Yu]{Fei Yu$^2$}
\address{Department of Mathematics, Xiamen University, Xiamen, China} 
\email{yu@xmu.edu.cn}
\thanks{${}^2$Partially supported by the Fundamental Research Funds for 
the Central Universities (No. 20720140526)}


\begin{abstract}
We compute the generating series  
for the intersection pairings between the total Chern classes of the tangent bundles 
of the Hilbert schemes of points on a smooth projective surface
and the Chern characters of tautological bundles over these Hilbert schemes.
Modulo the lower weight term, we verify Okounkov's conjecture \cite{Oko} 
connecting these Hilbert schemes and multiple $q$-zeta values. In addition,
this conjecture is completely proved when the surface is abelian.
We also determine some universal constants
in the sense of Boissi\' ere and Nieper-Wisskirchen \cite{Boi, BN} regarding
the total Chern classes of the tangent bundles of these Hilbert schemes.
The main approach of this paper is to use the set-up of Carlsson and Okounkov 
outlined in \cite{Car, CO} and the structure of the Chern character operators 
proved in \cite{LQW2}.
\end{abstract}

\maketitle

\section{\bf Introduction} 
\label{sect_intr}

In the region ${\rm Re} \, s > 1$, the Riemann zeta function is defined by 
$$
\zeta(s) = \sum_{n =1}^\infty {1 \over n^{s}}.
$$
The integers $s > 1$ give rise to a sequence of special values of the Riemann zeta function.
Multiple zeta values are series of the form
$$
\zeta(s_1, \ldots, s_k) = \sum_{n_1 > \cdots > n_k} {1 \over n_1^{s_1} \cdots n_k^{s_k}}
$$
where $n_1, \ldots, n_k$ denote positive integers, and $s_1, \ldots, s_k$ are
positive integers with $s_1 > 1$.
Multiple $q$-zeta values are $q$-deformations of $\zeta(s_1, \ldots, s_k)$, 
which may take different forms (see \cite{Bra1, Bra2, OT, Zud} for details). 
In \cite{Oko}, Okounkov proposed several interesting conjectures regarding 
multiple $q$-zeta values and Hilbert schemes of points. Motivated by these conjectures,
we compute in this paper the generating series  
for the intersection pairings between the total Chern classes of the tangent bundles 
of the Hilbert schemes of points on a smooth projective surface
and the Chern characters of tautological bundles over these Hilbert schemes.

Let $X$ be a smooth projective complex surface, 
and let $\Xn$ be the Hilbert scheme of $n$ points in $X$. 
A line bundle $L$ on $X$ induces a tautological rank-$n$ bundle $\Ln$ on $\Xn$.
Let $\ch_k(\Ln)$ be the $k$-th Chern character of $\Ln$.
Following Okounkov \cite{Oko}, we introduce the two generating series:
\begin{eqnarray}     
\big \langle \ch_{k_1}^{L_1} \cdots \ch_{k_N}^{L_N} \big \rangle
  &=&\sum_{n \ge 0} q^n \, \int_\Xn \ch_{k_1}(L_1^{[n]}) \cdots \ch_{k_N}(L_N^{[n]}) 
        \cdot c(T_\Xn)   \label{OkoChkN.1}   \\
\big \langle \ch_{k_1}^{L_1} \cdots \ch_{k_N}^{L_N} \big \rangle'
  &=&\big \langle \ch_{k_1}^{L_1} \cdots \ch_{k_N}^{L_N} \big \rangle/\langle 1 \rangle
      = (q; q)_\infty^{\chi(X)} \cdot \big \langle 
      \ch_{k_1}^{L_1} \cdots \ch_{k_N}^{L_N} \big \rangle  \label{OkoChkN.2} 
\end{eqnarray}
where $0 < q < 1$, $c\big (T_\Xn \big )$ is the total Chern class of 
the tangent bundle $T_\Xn$, $\chi(X)$ is the Euler characteristics, 
and $(a; q)_n = \prod_{i=0}^n (1-aq^i)$. 
In \cite{Car}, for $X = \C^2$ with a suitable $\C^*$-action and $L = \mathcal O_X$,
the series $\big \langle \ch_{k_1}^L \cdots \ch_{k_\ell}^L \big \rangle$ 
in the equivariant setting has been studied.
In \cite{Oko}, Okounkov proposed the following conjecture.

\begin{conjecture}   \label{OkoConj}
$\big \langle \ch_{k_1}^L \cdots \ch_{k_N}^L \big \rangle'$ is a multiple $q$-zeta value
of weight $\sum_{i=1}^N (k_i + 2)$.
\end{conjecture}

In this paper, we study Conjecture~\ref{OkoConj}. To state our result,
we introduce some definitions. For integers $n_i >0, w_i > 0$ and $p_i \ge 0$ 
with $1 \le i \le v$, define the {\it weight} of 
$\prod_{i=1}^v {q^{n_i w_ip_i} \over (1-q^{n_i})^{w_i}}$ to be
$\sum_{i=1}^v w_i$. For $k \ge 0$ and $\alpha \in H^*(X)$, 
define $\Theta^\alpha_k(q, z)$ to be the weight-$(k+2)$ multiple $q$-zeta value 
(with an additional variable $z$ inserted):
$$    
-\sum_{a, s_1, \ldots, s_a, b, t_1, \ldots, t_b \ge 1 
\atop \sum_{i=1}^a s_i + \sum_{j = 1}^b t_j = k+2} 
\big \langle (1_X - K_X)^{\sum_{i = 1}^a s_i}, \alpha \big \rangle
\prod_{i=1}^a {(-1)^{s_i} \over s_i!} \cdot \prod_{j=1}^b {1 \over t_j!} 
$$
\begin{eqnarray*}    
\cdot \sum_{n_1 > \cdots > n_a} \prod_{i= 1}^a {(q z)^{n_i s_i} \over (1-q^{n_i})^{s_i}}
\cdot \sum_{m_1 > \cdots > m_b} \prod_{j= 1}^b {z^{-m_jt_j} \over (1-q^{m_j})^{t_j}}
\end{eqnarray*} 
where $K_X$ and $1_X$ are the canonical class and fundamental class of $X$ respectively.
Let $\Coe_{z_1^0 \cdots z_N^0} (\cdot)$ denote 
the coefficient of $z_1^0 \cdots z_N^0$, $L$ also denote 
the first Chern class of the line bundle $L$, and $e_X$ be the Euler class of $X$.

\begin{theorem}   \label{Intro-AbelianSur}
Let $L_1, \ldots, L_N$ be line bundles over $X$, and $k_1, \ldots, k_N \ge 0$. Then, 
\begin{eqnarray}    \label{Intro-AbelianSur.0}
\big \langle \ch_{k_1}^{L_1} \cdots \ch_{k_N}^{L_N} \big \rangle'
= \Coe_{z_1^0 \cdots z_N^0} \left (\prod_{i=1}^N \Theta^{1_X}_{k_i}(q, z_i) \right ) + W,
\end{eqnarray}
and the lower weight term $W$ is an infinite linear combination of the expressions:
\begin{eqnarray*}      
\prod_{i=1}^u \left \langle K_X^{r_i}e_X^{r_i'}, L_1^{\ell_{i, 1}}  \cdots 
   L_N^{\ell_{i, N}} \right \rangle 
\cdot \prod_{i=1}^v {q^{n_i w_ip_i} \over (1-q^{n_i})^{w_i}}
\end{eqnarray*} 
where $\sum_{i=1}^v w_i < \sum_{i=1}^N (k_i + 2)$, 
and the integers $u, v$, $r_i, r_i', \ell_{i, j} \ge 0, 
n_i > 0, w_i >0, p_i \in \{0, 1\}$ depend only on $k_1, \ldots, k_N$.
Furthermore, all the coefficients of this linear combination are 
independent of $q, L_1, \ldots, L_N$ and $X$.
\end{theorem}

Theorem~\ref{Intro-AbelianSur} proves Conjecture~\ref{OkoConj}, 
modulo the lower weight term $W$. Note that the leading term 
$\Coe_{z_1^0 \cdots z_N^0} \left (\prod_{i=1}^N \Theta^{1_X}_{k_i}(q, z_i) \right )$
in $\big \langle \ch_{k_1}^{L_1} \cdots \ch_{k_N}^{L_N} \big \rangle'$ has weight 
$\sum_{i=1}^N (k_i + 2)$, and is a multiple of 
$\langle K_X, K_X \rangle^N$ whose coefficient depends only on $k_1, \ldots, k_N$
and is independent of the line bundles $L_1, \ldots, L_N$ and the surface $X$. 

In general, it is unclear how to organize the lower weight term $W$ in  
Theorem~\ref{Intro-AbelianSur} into multiple $q$-zeta values. On the other hand,
we have the following result which together with Theorem~\ref{Intro-AbelianSur} 
verifies Conjecture~\ref{OkoConj} when $X$ is an abelian surface.

\begin{theorem}   \label{Intro-KXEX=0}
Let $L_1, \ldots, L_N$ be line bundles over an abelian surface $X$, 
and $k_1, \ldots, k_N \ge 0$. Then, the lower weight term $W$ in \eqref{Intro-AbelianSur.0}
is a linear combination of the coefficients of $z_1^0 \cdots z_N^0$ in 
some multiple $q$-zeta values (with additional variables $z_1, \ldots, z_N$ inserted) 
of weights $< \sum_{i=1}^N (k_i +2)$. 
Moreover, the coefficients in this linear combination are independent of $q$.
\end{theorem}

We remark that some of the multiple $q$-zeta values mentioned in Theorem~\ref{Intro-KXEX=0} 
are in the generalized sense, i.e., in the following form:
$$
\sum_{n_1 > \cdots > n_\ell} \prod_{i=1}^{\ell} 
{(-n_i)^{w_i} q^{n_i p_i} f_i(z_1, \ldots, z_N)^{n_i} \over (1 - q^{n_i})^{w_i}} 
$$
where $0 \le p_i \le w_i$, and each $f_i(z_1, \ldots, z_N)$ is a monomial of 
$z_1^{\pm 1}, \ldots, z_N^{\pm 1}$.
We refer to \eqref{EllFactorial1} in the proof of Theorem~\ref{EX=0Fk1NAlpha1N} for more details.
As indicated in \cite{Oko}, the factors $(-n_i)^{w_i}$ in the above expression 
may be related to the operator $\displaystyle{q {{\rm d} \over {\rm d}q}}$.

The main idea in our proofs of Theorem~\ref{Intro-AbelianSur} and Theorem~\ref{Intro-KXEX=0}
is to use the structure of the Chern character operators proved in \cite{LQW2}
and the set-up of Carlsson and Okounkov in \cite{Car, CO}.
Let $G_k(\alpha, n)$ be the degree-$(2k + |\alpha|)$ component of \eqref{DefOfGGammaN},
and $\fG_{k}(\alpha)$ be the Chern character operator acting on the Fock space 
$\fock = \bigoplus_{n=0}^\infty \Hn$ via cup product by
$\bigoplus_{n=0}^\infty G_k(\alpha, n)$. Then,
$$
\ch_k (\Ln) = G_k(1_X, n) + G_{k-1}(L, n) + G_{k-2}(L^2/2, n)
$$
by the Grothendieck-Riemann-Roch Theorem.
So $\big \langle \ch_{k_1}^{L_1} \cdots \ch_{k_N}^{L_N} \big \rangle'$ is reduced to 
the series $F^{\alpha_1, \ldots, \alpha_N}_{k_1, \ldots, k_N}(q)$ 
defined by \eqref{F-generating}.
Let $\fL_1$ be the trivial line bundle on $X$ with a scaling action of $\C^*$ of 
character $1$~\footnote{Throughout the paper, we implicitly set $t = 1$ 
for the generator $t$ of the equivariant cohomology $H^*_{\C^*}({\rm pt})$ of a point.}.
Using the set-up in \cite{Car, CO}, we get
\begin{eqnarray*}    
F^{\alpha_1, \ldots, \alpha_N}_{k_1, \ldots, k_N}(q) = \Tr \, q^\fd \, W(\fL_1, z) 
\, \prod_{i=1}^N \fG_{k_i}(\alpha_i)
\end{eqnarray*}
where $W(\fL_1, z)$ is the vertex operator constructed in \cite{Car, CO}, 
and $\fd$ is the number-of-points operator (i.e., $\fd|_{H^*(\Xn)} = n \, {\rm Id}$).
The structure of the Chern character operators $\fG_{k_i}(\alpha_i)$ is given by
Theorem~\ref{char_th} which is proved in \cite{LQW2}. It implies that
the computation of $F^{\alpha_1, \ldots, \alpha_N}_{k_1, \ldots, k_N}(q)$ can be
further reduced to 
\begin{eqnarray}       \label{Intro-Trace}
\Tr \, q^\fd \, W(\fL_1, z) \, 
\prod_{i=1}^N {\fa_{\la^{(i)}}(\alpha_i) \over \la^{(i)}!}
\end{eqnarray}    
where $\la^{(i)}$ denotes a {\it generalized} partition which may also contain negative parts,
and $\la^{(i)}!$ and $\fa_{\la^{(i)}}(\alpha_i)$ are defined in 
Definition~\ref{partition}~(ii). The trace \eqref{Intro-Trace} is investigated via 
some standard but rather lengthy calculations.

As an application, our results enable us to determine some of the universal constants in 
$\displaystyle{\sum_n c\big (T_\Xn \big )} \, q^n$.
Let $C_i = {2i \choose i}/(i+1)$ be the Catalan number, and 
$\sigma_1(i) = \sum_{j|i} j$. By \cite{Boi, BN}.
there exist unique rational numbers $b_\mu, f_\mu, g_\mu, h_\mu$ 
depending only on the (usual) partitions $\mu$ such that 
$\displaystyle{\sum_n c\big (T_\Xn \big )} \, q^n$ is equal to 
$$
\exp \left ( \sum_{\mu} q^{|\mu|} \Big (b_\mu \fa_{-\mu}(1_X) 
+ f_\mu \fa_{-\mu}(e_X) + g_\mu \fa_{-\mu}(K_X) 
+ h_\mu \fa_{-\mu}(K_X^2) \Big ) \right ) \vac;
$$
in addition, $b_{2i} = 0$, 
$b_{2i-1} = (-1)^{i-1}C_{i-1}/(2i-1)$, 
$b_{(1^i)} = f_{(1^i)} = -g_{(1^i)} = \sigma_1(i)/i$, and $h_{(1^i)} = 0$.
In Theorem~\ref{bi1j}, we determine $b_{(i,1^j)}$ for $i \ge 2$ and $j \ge 0$.

The paper is organized as follows. In Sect.~\ref{sect_general}, 
we review the Heisenberg operators of Grojnowski and Nakajima,
and the structure of the Chern character operators.
In Sect.~\ref{sect_CO}, we recall the vertex operator of Carlsson and Okounkov.
In Sect.~\ref{sect_trace}, we compute the trace \eqref{Intro-Trace}.
Theorem~\ref{Intro-AbelianSur} and Theorem~\ref{Intro-KXEX=0} are proved 
in Sect.~\ref{sect_chk}. In Sect.~\ref{sect_app}, we determine the universal constants 
$b_{(i,1^j)}$ for $i \ge 2$ and $j \ge 0$.

\medskip\noindent
{\bf Convention.} All the (co)homology groups are in $\mathbb C$-coefficients unless 
otherwise specified. 
For $\alpha, \beta \in H^*(Y)$ where $Y$ is a smooth projective variety,
$\alpha \beta$ and $\alpha \cdot \beta$ denote the cup product $\alpha \cup \beta$, 
and $\langle \alpha, \beta \rangle$ denotes 
$\displaystyle{\int_Y \alpha \beta}$.

\bigskip\noindent
{\bf Acknowledgment.} The authors thank Professors Dan Edidin, Wei-ping Li and Weiqiang Wang 
for stimulating discussions and valuable helps. The second author also thanks the Mathematics 
Department of the University of Missouri for its hospitality during his visit in 
August 2015 as a Miller's Scholar.
\section{\bf Basics on Hilbert schemes of points on surfaces} 
\label{sect_general}

In this section, we will review some basic aspects of the Hilbert schemes of points on surfaces.
We will recall the definition of the Heisenberg operators of Grojnowski and Nakajima, 
and the structure of the Chern character operators.

Let $X$ be a smooth projective complex surface,
and $\Xn$ be the Hilbert scheme of $n$ points in $X$. 
An element in $\Xn$ is represented by a
length-$n$ $0$-dimensional closed subscheme $\xi$ of $X$. For $\xi
\in \Xn$, let $I_{\xi}$ be the corresponding sheaf of ideals. It
is well known that $\Xn$ is smooth. 
Define the universal codimension-$2$ subscheme:
\begin{eqnarray*}
{ \mathcal Z}_n=\{(\xi, x) \subset \Xn\times X \, | \, x\in
 {\rm Supp}{(\xi)}\}\subset \Xn\times X.
\end{eqnarray*}
Denote by $p_1$ and $p_2$ the projections of $\Xn \times X$ to
$\Xn$ and $X$ respectively. Let
\begin{eqnarray*}
\fock = \bigoplus_{n=0}^\infty \Hn
\end{eqnarray*}
be the direct sum of total cohomology groups of the Hilbert schemes $\Xn$.
For $m \ge 0$ and $n > 0$, let $Q^{[m,m]} = \emptyset$ and define
$Q^{[m+n,m]}$ to be the closed subset:
$$\{ (\xi, x, \eta) \in X^{[m+n]} \times X \times X^{[m]} \, | \,
\xi \supset \eta \text{ and } \mbox{Supp}(I_\eta/I_\xi) = \{ x \}
\}.$$

We recall Nakajima's definition of the Heisenberg operators \cite{Nak}. 
Let $\alpha \in H^*(X)$. Set $\mathfrak a_0(\alpha) =0$.
For $n > 0$, the operator $\mathfrak
a_{-n}(\alpha) \in \End(\fock)$ is
defined by
$$
\mathfrak a_{-n}(\alpha)(a) = \w{p}_{1*}([Q^{[m+n,m]}] \cdot
\w{\rho}^*\alpha \cdot \w{p}_2^*a)
$$
for $a \in H^*(X^{[m]})$, where $\w{p}_1, \w{\rho},
\w{p}_2$ are the projections of $X^{[m+n]} \times X \times
X^{[m]}$ to $X^{[m+n]}, X, X^{[m]}$ respectively. Define
$\mathfrak a_{n}(\alpha) \in \End(\fock)$ to be $(-1)^n$ times the
operator obtained from the definition of $\mathfrak
a_{-n}(\alpha)$ by switching the roles of $\w{p}_1$ and $\w{p}_2$. 
We often refer to $\mathfrak a_{-n}(\alpha)$ (resp. $\mathfrak a_n(\alpha)$) 
as the {\em creation} (resp. {\em annihilation})~operator. 
The following is from \cite{Nak, Gro}. Our convention of the sign follows \cite{LQW2}.

\begin{theorem} \label{commutator}
The operators $\mathfrak a_n(\alpha)$ satisfy
the commutation relation:
\begin{eqnarray*}
\displaystyle{[\mathfrak a_m(\alpha), \mathfrak a_n(\beta)] = -m
\; \delta_{m,-n} \cdot \langle \alpha, \beta \rangle \cdot {\rm Id}_{\fock}}.
\end{eqnarray*}
The space $\fock$ is an irreducible module over the Heisenberg
algebra generated by the operators $\mathfrak a_n(\alpha)$ with a
highest~weight~vector $\vac=1 \in H^0(X^{[0]}) \cong \C$.
\end{theorem}

The Lie bracket in the above theorem is understood in the super
sense according to the parity of the cohomology degrees of the
cohomology classes involved. It follows from
Theorem~\ref{commutator} that the space $\fock$ is linearly
spanned by all the Heisenberg monomials $\mathfrak
a_{n_1}(\alpha_1) \cdots \mathfrak a_{n_k}(\alpha_k) \vac$
where $k \ge 0$ and $n_1, \ldots, n_k < 0$.

\begin{definition} \label{partition}
\begin{enumerate}
\item[{\rm (i)}] 
Let $\alpha \in H^*(X)$ and $k \ge 1$. Define $\tau_{k*}: H^*(X) \to H^*(X^k)$ 
to be the linear map induced by the diagonal embedding $\tau_k: X \to X^k$, and
$$
(\mathfrak a_{m_1} \cdots \mathfrak a_{m_k})(\alpha)
= \mathfrak a_{m_1} \cdots \mathfrak a_{m_k}(\tau_{*k}\alpha)
= \sum_j \mathfrak a_{m_1}(\alpha_{j,1}) \cdots \mathfrak a_{m_k}(\alpha_{j,k})
$$ 
when $\tau_{k*}\alpha = \sum_j \alpha_{j,1} \otimes \cdots 
\otimes \alpha_{j, k}$ via the K\"unneth decomposition of $H^*(X^k)$.

\item[{\rm (ii)}]
Let $\lambda = \big (\cdots
(-2)^{m_{-2}}(-1)^{m_{-1}} 1^{m_1}2^{m_2} \cdots \big )$ be a {\em
generalized partition} of the integer $n = \sum_i i m_i$ whose
part $i\in \Z$ has multiplicity $m_i$. Define $\ell(\lambda) =
\sum_i m_i$, $|\lambda| = \sum_i i m_i = n$, $\lambsq  = \sum_i
i^2 m_i$, $\lambda! = \prod_i m_i!$, and
\begin{eqnarray*}
\mathfrak a_{\lambda}(\alpha) = \left ( \prod_i (\mathfrak
a_i)^{m_i} \right ) (\alpha)
\end{eqnarray*}
where the product $\prod_i (\mathfrak
a_i)^{m_i} $ is understood to be
$\cdots \mathfrak a_{-2}^{m_{-2}} \mathfrak a_{-1}^{m_{-1}}
 \mathfrak a_{1}^{m_{1}} \mathfrak a_{2}^{m_{2}}\cdots$.
The set of all generalized partitions is denoted by $\Wcp$.

\item[{\rm (iii)}] A generalized partition becomes a {\em partition}
in the usual sense if the multiplicity $m_ i = 0$ for all $i < 0$. 
The set of all partitions is denoted by $\cp$.
\end{enumerate}
\end{definition}

For $n > 0$ and a homogeneous class $\alpha \in H^*(X)$, let
$|\alpha| = s$ if $\alpha \in H^s(X)$, and let $G_k(\alpha, n)$ be
the homogeneous component in $H^{|\alpha|+2k}(\Xn)$ of
\begin{eqnarray}    \label{DefOfGGammaN}
 G(\alpha, n) = p_{1*}(\ch({\mathcal O}_{{\mathcal Z}_n}) \cdot p_2^*\alpha
\cdot p_2^*{\rm td}(X) ) \in \Hn
\end{eqnarray}
where $\ch({\mathcal O}_{{\mathcal Z}_n})$ denotes the Chern
character of ${\mathcal O}_{{\mathcal Z}_n}$
and ${\rm td}(X) $ denotes the Todd class. We extend the notion $G_k(\alpha,
n)$ linearly to an arbitrary class $\alpha \in H^*(X)$, and set $G(\alpha, 0) =0$. 
It was proved in \cite{LQW1} that the cohomology ring of $\Xn$ is
generated by the classes $G_{k}(\alpha, n)$ where $0 \le k < n$
and $\alpha$ runs over a linear basis of $H^*(X)$. 
The {\it Chern character operator} ${\mathfrak G}_k(\alpha) \in
\End({\fock})$ is the operator acting on $H^*(\Xn)$ by the cup product with $G_k(\alpha, n)$. 
The following is from \cite{LQW2}.

\begin{theorem} \label{char_th}
Let $k \ge 0$ and $\alpha\in H^*(X)$. Then, $\mathfrak G_k(\alpha)$ is equal to
\begin{eqnarray*}
& &- \sum_{\ell(\lambda) = k+2, |\lambda|=0}
   {1 \over \lambda!} \mathfrak a_{\lambda}(\alpha)
   + \sum_{\ell(\lambda) = k, |\lambda|=0}
   {\lambsq - 2 \over 24\lambda!}
   \mathfrak a_{\lambda}(e_X \alpha)  \\
&+&\sum_{\ell(\lambda) = k+1, |\lambda|=0} {g_{1, \lambda} \over \la!}
   \mathfrak a_{\lambda}(K_X \alpha)
   + \sum_{\ell(\lambda) = k, |\lambda|=0} {g_{2, \lambda} \over \la!}
   \mathfrak a_{\lambda}(K_X^2 \alpha)
\end{eqnarray*}
where all the numbers $g_{1, \lambda}$ and $g_{2, \lambda}$ are independent of 
$X$ and $\alpha$.
\end{theorem}

For $\alpha_1, \ldots, \alpha_N \in H^*(X)$ and integers $k_1, \ldots, k_N \ge 0$, 
define the series
\begin{eqnarray}    \label{F-generating}
F^{\alpha_1, \ldots, \alpha_N}_{k_1, \ldots, k_N}(q) = 
\sum_n q^n \int_\Xn \left ( \prod_{i=1}^N G_{k_i}(\alpha_i, n) \right ) c\big (T_\Xn \big ).
\end{eqnarray}
In view of G\" ottsche's Theorem in \cite{Got}, we have
$F(q) = (q; q)_\infty^{-\chi(X)}$.

The following is from \cite{LQW3} and will be used throughout the paper.

\begin{lemma} \label{tau_k_tau_{k-1}}
Let $k, s \ge 1$, $n_1, \ldots, n_k, m_1, \ldots, m_s \in \Z$, and
$\alpha, \beta \in H^*(X)$.
\begin{enumerate}
\item[{\rm (i)}] 
The commutator $[(\mathfrak a_{n_1} \cdots \mathfrak a_{n_{k}}) (\alpha),
(\mathfrak a_{m_1} \cdots \mathfrak a_{m_{s}})(\beta)]$ is
equal to
\begin{eqnarray*}
-\sum_{t=1}^k \sum_{j=1}^s n_t \delta_{n_t,-m_j} \cdot \left (
\prod_{\ell=1}^{j-1} \mathfrak a_{m_\ell} \prod_{1 \le u \le k, u
\ne t} \mathfrak a_{n_u} \prod_{\ell=j+1}^{s} \mathfrak a_{m_\ell}
\right )(\alpha\beta).
\end{eqnarray*}

\item[{\rm (ii)}] Let $j$ satisfy $1 \le j < k$. Then,
$(\mathfrak a_{n_1} \cdots \mathfrak a_{n_k})(\alpha)$ is equal to
\begin{eqnarray*}
\left ( \prod_{1 \le s < j} \mathfrak a_{n_s} \cdot \mathfrak
a_{n_{j+1}} \mathfrak a_{n_{j}} \cdot \prod_{j+1 < s \le k}
\mathfrak a_{n_s} \right ) (\alpha) - n_j
\delta_{n_j,-n_{j+1}} \left ( \prod_{1 \le s \le k \atop s \ne j, j+1} 
\mathfrak a_{n_s} \right )(e_X\alpha).
\end{eqnarray*}
\end{enumerate}
\end{lemma}
\section{\bf The vertex operators of Carlsson and Okounkov} 
\label{sect_CO}

In this section, we will recall the vertex operators constructed in \cite{CO, Car}, 
and use them to rewrite the generating 
series $F^{\alpha_1, \ldots, \alpha_N}_{k_1, \ldots, k_N}(q)$
defined in \eqref{F-generating}.

Let $L$ be a line bundle over the smooth projective surface $X$. 
Let $\mathbb E_L$ be the virtual vector bundle on $X^{[k]} \times X^{[\ell]}$
whose fiber at $(I, J) \in X^{[k]} \times X^{[\ell]}$ is given by
$$
\mathbb E_L|_{(I,J)} = \chi(\mathcal O, L) - \chi(J, I \otimes L).
$$
Let $\fL_m$ be the trivial line bundle on $X$ with a scaling action of $\C^*$ of 
character $m$, and let $\Delta_n$ be the diagonal in $\Xn \times \Xn$. Then, 
\begin{eqnarray}   \label{ResOfEToD}
\mathbb E_{\fL_m}|_{\Delta_n} = T_{\Xn, m},
\end{eqnarray}
the tangent bundle $T_\Xn$ with a scaling action of $\C^*$ of character $m$. 
By abusing notations, we also use $L$ to denote its first Chern class. Put
\begin{eqnarray}   \label{GammaLz}
\Gamma_{\pm}(L, z) = \exp \left ( \sum_{n>0} {z^{\mp n} \over n} \fa_{\pm n}(L) \right ).
\end{eqnarray}

\begin{remark}  \label{SignDiff}
There is a sign difference between the Heisenberg commutation relations used in 
\cite{Car} (see p.3 there) and in this paper (see Theorem~\ref{commutator}).
So for $n > 0$, our Heisenberg operators $\fa_{-n}(L)$ and $\fa_{n}(-L)$
are equal to the Heisenberg operators $\fa_{-n}(L)$ and $\fa_{n}(L)$ in \cite{Car}.
Accordingly, our operators $\Gamma_-(L, z)$ and $\Gamma_+(-L, z)$ are equal to
the operators $\Gamma_-(L, z)$ and $\Gamma_+(L, z)$ in \cite{Car}.
\end{remark}

The following commutation relations can be found in \cite{Car} (see Remark~\ref{SignDiff}):
\begin{eqnarray}   
&[\Gamma_+(L, x), \Gamma_+(L', y)] = [\Gamma_-(L, x), \Gamma_-(L', y)] = 0,&
                                           \label{CarLemma5.1}          \\
&\Gamma_+(L, x)\Gamma_-(L', y) = (1-y/x)^{\langle L, L' \rangle} \,
    \Gamma_-(L', y) \Gamma_+(L, x).&     \label{CarLemma5.2}
\end{eqnarray}

Let $W(L, z): \fock \to \fock$ be the vertex operator constructed in \cite{CO, Car} 
where $z$ is a formal variable. By \cite{Car}, $W(L, z)$ is defined via the pairing
\begin{eqnarray}   \label{def-WLz}
\langle W(L, z) \eta, \xi \rangle = \int_{X^{[k]} \times X^{[\ell]}}
(\eta \otimes \xi) \, c_{k+\ell}(\mathbb E_L)
\end{eqnarray}
for $\eta \in H^*(X^{[k]})$ and $\xi \in H^*(X^{[\ell]})$.
The main result in \cite{Car} is (see Remark~\ref{SignDiff}):
\begin{eqnarray}   \label{WLz}
W(L, z) = \Gamma_-(L-K_X, z) \, \Gamma_+(-L, z).
\end{eqnarray}

\begin{lemma}  \label{FtoW}
Let $\fd$ be the number-of-points operator, i.e., $\fd|_{H^*(\Xn)} = n \, {\rm Id}$. Then,
\begin{eqnarray}   \label{FtoW.0}
F^{\alpha_1, \ldots, \alpha_N}_{k_1, \ldots, k_N}(q) = \Tr \, q^\fd \, W(\fL_1, z) 
\, \prod_{i=1}^N \fG_{k_i}(\alpha_i).
\end{eqnarray}
\end{lemma}
\noindent
{\it Proof.}
We will show that the coefficients of $q^n$ on both sides of \eqref{FtoW.0} are equal.
Let $\{e_j\}_j$ be a linear basis of $H^*(\Xn)$. 
Then the fundamental class of the diagonal $\Delta_n$ in $\Xn \times \Xn$ is given by 
$
[\Delta_n] = \sum_j (-1)^{|e_j|} \, e_j \otimes e_j^*
$
where $\{e_j^*\}_j$ is the linear basis of $H^*(\Xn)$ dual to $\{e_j\}_j$ in the sense 
that $\langle e_j, e_{j'}^* \rangle = \delta_{j, j'}$.
By the definitions of $W(L, z)$ and $\fG_{k}(\alpha)$, 
$\displaystyle{\Tr \, q^n W(\fL_1, z) \, \prod_{i=1}^N \fG_{k_i}(\alpha_i)}$ is equal to
\begin{eqnarray*} 
& &q^n \sum_j (-1)^{|e_j|} \, \left \langle W(\fL_1, z) 
         \left (\prod_{i=1}^N \fG_{k_i}(\alpha_i) \right ) e_j, e_j^* \right \rangle  \\
&=&q^n \sum_j (-1)^{|e_j|} \, \int_{\Xn \times \Xn} \left (\left (\prod_{i=1}^N \fG_{k_i}(\alpha_i) \right )
             e_j \otimes e_j^* \right ) \, c_{2n}(\mathbb E_{\fL_1})  \\
&=&q^n \sum_j (-1)^{|e_j|} \, \int_{\Xn \times \Xn} \left (\left (\prod_{i=1}^N G_{k_i}(\alpha_i, n) \right )
             e_j \otimes e_j^* \right ) \, c_{2n}(\mathbb E_{\fL_1}) \\
&=&q^n \sum_j (-1)^{|e_j|} \, \int_{\Xn \times \Xn} (e_j \otimes e_j^*) \, p_1^*\left (\prod_{i=1}^N 
              G_{k_i}(\alpha_i, n) \right ) \, c_{2n}(\mathbb E_{\fL_1})  \\
&=&q^n \int_{\Xn \times \Xn} [\Delta_n] \, p_1^*\left (\prod_{i=1}^N 
              G_{k_i}(\alpha_i, n) \right ) \, c_{2n}(\mathbb E_{\fL_1})
\end{eqnarray*}
where $p_1: \Xn \times \Xn \to \Xn$ denotes the first projection.
By \eqref{ResOfEToD}, we have $c_{2n}\big ( \mathbb E_{\fL_1} \big )|_{\Delta_n} 
= c\big (T_\Xn \big )$.
Here and below, we implicitly set $t = 1$ for the generator $t$ of 
the equivariant cohomology $H^*_{\C^*}({\rm pt})$ of a point. Therefore, 
\begin{equation}
\Tr \, q^n  W(\fL_1, z) \, \prod_{i=1}^N \fG_{k_i}(\alpha_i)
= q^n \int_{\Xn} \left (\prod_{i=1}^N 
              G_{k_i}(\alpha_i, n) \right ) \, c\big (T_\Xn \big ).
\tag*{$\qed$}
\end{equation}

\section{\bf The trace $\displaystyle{\Tr \, q^\fd \, W(\fL_1, z) \, 
\prod_{i=1}^N {\fa_{\la^{(i)}}(\alpha_i) \over \la^{(i)}!}}$ and 
the series $F^{\alpha_1, \ldots, \alpha_N}_{k_1, \ldots, k_N}(q)$} 
\label{sect_trace}

In this section, we will first determine the structure of 
$\displaystyle{\Tr \, q^\fd \, W(\fL_1, z) \, 
\prod_{i=1}^N {\fa_{\la^{(i)}}(\alpha_i) \over \la^{(i)}!}}$. 
Then, the structure of the generating series 
$F^{\alpha_1, \ldots, \alpha_N}_{k_1, \ldots, k_N}(q)$ will
follow from Lemma~\ref{FtoW}, Theorem~\ref{char_th} and the structure of
$\displaystyle{\Tr \, q^\fd \, W(\fL_1, z) \, 
\prod_{i=1}^N {\fa_{\la^{(i)}}(\alpha_i) \over \la^{(i)}!}}$. 

We begin with four technical lemmas. To explain the ideas behind these lemmas, 
note from \eqref{WLz} that $\displaystyle{\Tr \, q^\fd \, W(\fL_1, z) \, 
\prod_{i=1}^N {\fa_{\la^{(i)}}(\alpha_i) \over \la^{(i)}!}}$ is equal to
\begin{eqnarray}    \label{sect_traceIdeas}
\Tr \, q^\fd \, \Gamma_-(\fL_1-K_X, z) \, \Gamma_+(-\fL_1, z) \, 
     \prod_{i=1}^N {\fa_{\la^{(i)}}(\alpha_i) \over \la^{(i)}!}.
\end{eqnarray}
Lemma~\ref{CommJL-} deals with the commutator between 
$\displaystyle{{\fa_\la(\alpha) \over \la!}}$ and 
$\displaystyle{\exp \left ( {z^n \over n} \fa_{-n}(\gamma) \right )}$.
It enables us in Lemma~\ref{LemmaTrace} 
to eliminate $\Gamma_-(\fL_1-K_X, z)$ from \eqref{sect_traceIdeas},
and allows us in Lemma~\ref{20150806957} to eliminate $\Gamma_+(-\fL_1, z)$ 
from \eqref{sect_traceIdeas}. Lemma~\ref{Trqdjpx} determines the structure 
of $\displaystyle{\Tr \, q^\fd \,\prod_{i=1}^N 
{\fa_{\la^{(i)}}(\alpha_i) \over \la^{(i)}!}}$.
The proofs of these lemmas are standard but lengthy.

Recall from Definition~\ref{partition}~(ii) that $\Wcp$ denotes
the set of generalized partitions. 
If $\lambda = \big (\cdots (-2)^{s_{-2}}(-1)^{s_{-1}} 1^{s_1}2^{s_2} \cdots \big )$
and $\mu = \big (\cdots (-2)^{t_{-2}}(-1)^{t_{-1}} 1^{t_1}2^{t_2} \cdots \big )$, let
$$
\la - \mu 
= \big (\cdots (-2)^{s_{-2} - t_{-2}}(-1)^{s_{-1}-t_{-1}} 1^{s_1-t_1}2^{s_2-t_2} \cdots \big )
$$
with the convention that $\la - \mu = \emptyset$ if $s_i < t_i$ for some $i$.

\begin{lemma}  \label{CommJL-}
Let $\la \in \Wcp$. Assume that $\gamma \in H^{\rm even}(X)$. Then, 
\begin{eqnarray}
   {\fa_\la(\alpha) \over \la!} \exp \left ( {z^n \over n} \fa_{-n}(\gamma) \right )
&=&\exp \left ( {z^n \over n} \fa_{-n}(\gamma) \right ) \cdot \sum_{i \ge 0} {(-z^n)^i \over i!} 
   {\fa_{\la - (n^i)}(\gamma^i\alpha) \over ({\la - (n^i)})!},   \label{CommJL-.01}  \\
   \exp \left ( {z^n \over n} \fa_{n}(\gamma) \right ) \cdot {\fa_\la(\alpha) \over \la!} 
&=&\sum_{i \ge 0} {(-z^n)^i \over i!} 
   {\fa_{\la - ((-n)^i)}(\gamma^i\alpha) \over ({\la - ((-n)^i)})!}
   \exp \left ( {z^n \over n} \fa_{n}(\gamma) \right ).   \label{CommJL-.02}
\end{eqnarray}
\end{lemma}
\begin{proof}
Note that the adjoint of $\fa_m(\beta)$ is equal to $(-1)^m \fa_{-m}(\beta)$.
So \eqref{CommJL-.02} follows from \eqref{CommJL-.01} by taking adjoint on both sides of 
\eqref{CommJL-.01} and by making suitable adjustments.
To prove \eqref{CommJL-.01}, put $A = \displaystyle{{\fa_\la(\alpha) \over \la!}
\exp \left ( {z^n \over n} \fa_{-n}(\gamma) \right )}$. Then,
\begin{eqnarray*}   
   A
&=&{\fa_\la(\alpha) \over \la!}
   \sum_{t \ge 0} {1 \over t!} \left ({z^n \over n} \fa_{-n}(\gamma) \right )^t   \\
&=&{1 \over \la!} \sum_{t \ge 0} {1 \over t!} \left ({z^n \over n} \right )^t
   \sum_{i=0}^t {t \choose i} 
   \big ( \fa_{-n}(\gamma) \big )^{t-i}   \big [\cdots [ \fa_\la(\alpha), \underbrace{\fa_{-n}(\gamma)], 
   \ldots, \fa_{-n}(\gamma)}_{i \,\, {\rm times}} \big ].
\end{eqnarray*}
Let $\lambda = \big (\cdots (-2)^{s_{-2}}(-1)^{s_{-1}} 1^{s_1}2^{s_2} \cdots \big )$.
We conclude from Lemma~\ref{tau_k_tau_{k-1}}~(i) that the commutator
$\big [\cdots [ \fa_\la(\alpha), \underbrace{\fa_{-n}(\gamma)], 
   \ldots, \fa_{-n}(\gamma)}_{i \,\, {\rm times}} \big ]$ is equal to
$$
s_n(s_n-1)\cdots (s_n+1-i) \, (-n)^i \, \fa_{\la - (n^i)}(\gamma^i\alpha)
$$
where by our convention, $\la - (n^i) = \emptyset$ if $s_n < i$. So $A$ is equal to
\begin{eqnarray*}   
& &{1 \over \la!} \sum_{t \ge 0} {1 \over t!} \left ({z^n \over n} \right )^t
   \sum_{i=0}^t {t \choose i} 
   \big ( \fa_{-n}(\gamma) \big )^{t-i}  \cdot     \\
& &\quad \quad \cdot s_n(s_n-1)\cdots (s_n+1-i) \, (-n)^i \, \fa_{\la - (n^i)}(\gamma^i\alpha).
\end{eqnarray*}
Simplifying this, we complete the proof of our formula \eqref{CommJL-.01}.
\end{proof}

Let $\Wcp_+ = \cp$ be the subset of $\Wcp$ consisting of the usual partitions,
and $\Wcp_-$ be the subset of $\Wcp$ consisting of generalized partitions
of the form $\big (\cdots (-2)^{s_{-2}}(-1)^{s_{-1}} \big )$.

\begin{lemma}  \label{LemmaTrace}
Let $\la^{(1)}, \ldots, \la^{(N)} \in \Wcp$ be generalized partitions, 
and $\alpha_1, \ldots, \alpha_N \in H^*(X)$. 
Then, the trace $\displaystyle{\Tr \, q^\fd \, W(\fL_1, z) \, 
\prod_{i=1}^N {\fa_{\la^{(i)}}(\alpha_i) \over \la^{(i)}!}}$ is equal to
$$
\sum_{\mu^{(i, s)} \in \Wcp_+ \atop 1 \le i \le N, s \ge 1} \,
\prod_{1 \le i \le N \atop s, n \ge 1} {(-(zq^s)^n)^{m^{(i, s)}_{n}} \over m^{(i, s)}_{n}!} \cdot 
$$
\begin{eqnarray}    \label{LemmaTrace.0}
\cdot \Tr \, q^\fd \, \Gamma_+(-\fL_1, z) \prod_{i=1}^N 
{\fa_{\la^{(i)}-\sum_{s \ge 1} \mu^{(i, s)}}
   \big ((1_X - K_X)^{\sum_{s, n \ge 1} m^{(i, s)}_{n}}\alpha_i \big ) 
   \over \big ( \la^{(i)}-\sum_{s \ge 1} \mu^{(i, s)} \big )!}.
\end{eqnarray}
where $\mu^{(i, s)} = \big (1^{m^{(i, s)}_{1}}\cdots n^{m^{(i, s)}_{n}}\cdots \big )
\in \Wcp_+$ for $1 \le i \le N$ and $s \ge 1$.
\end{lemma}
\begin{proof}
For simplicity, put $Q_1 = \displaystyle{\Tr \, q^\fd \, W(\fL_1, z) \, 
\prod_{i=1}^N {\fa_{\la^{(i)}}(\alpha_i) \over \la^{(i)}!}}$. By \eqref{WLz},
\begin{eqnarray}
   Q_1
&=&\Tr \, q^\fd \, \Gamma_-(\fL_1-K_X, z) \, \Gamma_+(-\fL_1, z) \, 
     \prod_{i=1}^N {\fa_{\la^{(i)}}(\alpha_i) \over \la^{(i)}!}    \nonumber     \\
&=&\Tr \, \Gamma_-(\fL_1-K_X, zq) \, q^\fd \, \Gamma_+(-\fL_1, z) \, 
     \prod_{i=1}^N {\fa_{\la^{(i)}}(\alpha_i) \over \la^{(i)}!}    \label{Repeat.1}     \\
&=&\Tr \, q^\fd \, \Gamma_+(-\fL_1, z) \prod_{i=1}^N {\fa_{\la^{(i)}}(\alpha_i) \over \la^{(i)}!}
     \cdot \Gamma_-(\fL_1-K_X, zq).                     \nonumber  
\end{eqnarray}
By \eqref{GammaLz} and applying \eqref{CommJL-.01} repeatedly, we obtain
\begin{eqnarray*}
& &\prod_{i=1}^N {\fa_{\la^{(i)}}(\alpha_i) \over \la^{(i)}!} \cdot \Gamma_-(\fL_1-K_X, zq)     \\
&=&\prod_{i=1}^N {\fa_{\la^{(i)}}(\alpha_i) \over \la^{(i)}!} \cdot 
   \exp \left ( \sum_{n>0} {(zq)^{n} \over n} \fa_{-n}(\fL_1-K_X) \right )       \\
&=&\Gamma_-(\fL_1-K_X, zq) \sum_{\mu^{(i, 1)} \in \Wcp_+ \atop 1 \le i \le N}
   \prod_{1 \le i \le N \atop n \ge 1} {(-(zq)^n)^{m_{n}^{(i,1)}} \over m_{n}^{(i,1)}!} \cdot      \\
& &\quad \quad \cdot \prod_{i=1}^N {\fa_{\la^{(i)} - \mu^{(i, 1)}}
   \big ((1_X - K_X)^{\sum_{n \ge 1}m^{(i, 1)}_{n}}\alpha_i \big ) 
   \over \big ( \la^{(i)} - \mu^{(i, 1)} \big )!}.  
\end{eqnarray*}
where $\mu^{(i, 1)} = \big (1^{m^{(i, 1)}_{1}}\cdots n^{m^{(i, 1)}_{n}}\cdots \big )$.
Therefore, $Q_1$ is equal to
$$
\Tr \, q^\fd \, \Gamma_+(-\fL_1, z)  \Gamma_-(\fL_1-K_X, zq)  
\sum_{\mu^{(i, 1)} \in \Wcp_+ \atop 1 \le i \le N}
   \prod_{1 \le i \le N \atop n \ge 1} {(-(zq)^n)^{m_{n}^{(i,1)}} \over m_{n}^{(i,1)}!} \cdot 
$$
$$
\cdot \prod_{i=1}^N {\fa_{\la^{(i)} - \mu^{(i, 1)}}
   \big ((1_X - K_X)^{\sum_{n \ge 1}m^{(i, 1)}_{n}}\alpha_i \big ) 
   \over \big ( \la^{(i)} - \mu^{(i, 1)} \big )!}.
$$
Since $\langle \fL_1, \fL_1-K_X \rangle = 0$, we see from \eqref{CarLemma5.2} that 
$Q_1$ is equal to
$$
\Tr \, q^\fd \, \Gamma_-(\fL_1-K_X, zq) \Gamma_+(-\fL_1, z) \cdot
\sum_{\mu^{(i, 1)} \in \Wcp_+ \atop 1 \le i \le N}
   \prod_{1 \le i \le N \atop n \ge 1} {(-(zq)^n)^{m_{n}^{(i,1)}} \over m_{n}^{(i,1)}!} \cdot 
$$
$$
\cdot \prod_{i=1}^N {\fa_{\la^{(i)} - \mu^{(i, 1)}}
   \big ((1_X - K_X)^{\sum_{n \ge 1}m^{(i, 1)}_{n}}\alpha_i \big ) 
   \over \big ( \la^{(i)} - \mu^{(i, 1)} \big )!}.
$$
Repeat the above process beginning at line \eqref{Repeat.1} $s$ times. Then, 
$Q_1$ is equals  to
$$
\Tr \, q^\fd \, \Gamma_-(\fL_1-K_X, zq^s) \Gamma_+(-\fL_1, z) \cdot
\sum_{\mu^{(i, r)} \in \Wcp_+ \atop 1 \le i \le N, 1 \le r \le s}
   \prod_{1 \le i \le N \atop 1 \le r \le s, n \ge 1} 
   {(-(zq^r)^n)^{m_{n}^{(i,r)}} \over m_{n}^{(i,r)}!} \cdot 
$$
$$
\cdot \prod_{i=1}^N {\fa_{\la^{(i)} - \sum_{r=1}^s \mu^{(i, r)}}
   \big ((1_X - K_X)^{\sum_{r=1}^s \sum_{n \ge 1}m^{(i, r)}_{n}}\alpha_i \big ) 
   \over \big ( \la^{(i)} - \sum_{r=1}^s \mu^{(i, r)} \big )!}
$$
where $\mu^{(i, r)} = \big (1^{m^{(i, r)}_{1}}\cdots n^{m^{(i, r)}_{n}}\cdots \big )$.
Letting $s \to +\infty$ proves our lemma.
\end{proof}

\begin{lemma}  \label{20150806957}
Let $\w \la^{(1)}, \ldots, \w \la^{(N)} \in \Wcp$ be generalized partitions, 
and $\w \alpha_1, \ldots, \w \alpha_N \in H^*(X)$. 
Then, the trace $\displaystyle{\Tr \, q^\fd \, \Gamma_+(-\fL_1, z) 
\prod_{i=1}^N {\fa_{\w \la^{(i)}}(\w \alpha_i) \over \w \la^{(i)}!}}$ is equal to
\begin{eqnarray*}
\sum_{\sum_{i=1}^N(|\w \la^{(i)}| - \sum_{t \ge 1} |\w \mu^{(i, t)}|) = 0} 
     \prod_{1 \le i \le N \atop t, n \ge 1} {(z^{-1}q^{t-1})^{n \w m^{(i,t)}_n} 
     \over {\w m^{(i,t)}_n}!} 
\cdot     \Tr \, q^\fd \prod_{i=1}^N 
     {\fa_{\w \la^{(i)} - \sum_{t \ge 1} \w \mu^{(i, t)}}(\w \alpha_i) \over
     \big (\w \la^{(i)} - \sum_{t \ge 1} \w \mu^{(i, t)} \big )!}
\end{eqnarray*}
where $\w \mu^{(i, t)} = \big (\cdots (-n)^{\w m^{(i, t)}_{n}}\cdots (-1)^{\w m^{(i, t)}_{1}}\big )
\in \Wcp_-$ for $1 \le i \le N$ and $t \ge 1$.
\end{lemma}
\begin{proof}
For simplicity, put $Q_2 = \displaystyle{\Tr \, q^\fd \, \Gamma_+(-\fL_1, z) 
\prod_{i=1}^N {\fa_{\w \la^{(i)}}(\w \alpha_i) \over \w \la^{(i)}!}}$. By \eqref{GammaLz},
$$
Q_2 = \Tr \, q^\fd \, \exp \left ( \sum_{n>0} {z^{-n} \over n} \fa_n(-\fL_1) \right )  \, 
\prod_{i=1}^N {\fa_{\w \la^{(i)}}(\w \alpha_i) \over \w \la^{(i)}!}.
$$
Applying \eqref{CommJL-.02} repeatedly, we see that $Q_2$ is equal to
$$
\sum_{\w \mu^{(i, 1)} \in \Wcp_- \atop 1 \le i \le N} 
     \prod_{1 \le i \le N \atop n \ge 1} {z^{-n \w m^{(i,1)}_n} \over {\w m^{(i,1)}_n}!} 
     \cdot \Tr \, q^\fd \prod_{i=1}^N 
     {\fa_{\w \la^{(i)} - \w \mu^{(i, 1)}}(\w \alpha_i) \over
     \big (\w \la^{(i)} - \w \mu^{(i, 1)} \big )!} \cdot
     \Gamma_+(-\fL_1, z) 
$$
where $\w \mu^{(i, 1)} = \big (\cdots (-n)^{\w m^{(i,1)}_n} \cdots (-1)^{\w m^{(i,1)}_1} \big )
\in \Wcp_-$. Now $Q_2$ is equal to
\begin{eqnarray*}
& &\sum_{\w \mu^{(i, 1)} \in \Wcp_- \atop 1 \le i \le N} 
     \prod_{1 \le i \le N \atop n \ge 1} {z^{-n \w m^{(i,1)}_n} \over {\w m^{(i,1)}_n}!} 
     \cdot q^{\sum_{i=1}^N(|\w \mu^{(i, 1)}| - |\w \la^{(i)}|)} 
     \Tr \, \prod_{i=1}^N 
     {\fa_{\w \la^{(i)} - \w \mu^{(i, 1)}}(\w \alpha_i) \over
     \big (\w \la^{(i)} - \w \mu^{(i, 1)} \big )!} \cdot
     q^\fd \Gamma_+(-\fL_1, z) \\
&=&\sum_{\w \mu^{(i, 1)} \in \Wcp_- \atop 1 \le i \le N} 
     \prod_{1 \le i \le N \atop n \ge 1} {z^{-n \w m^{(i,1)}_n} \over {\w m^{(i,1)}_n}!} 
     \cdot q^{\sum_{i=1}^N(|\w \mu^{(i, 1)}| - |\w \la^{(i)}|)} 
     \Tr \, q^\fd \Gamma_+(-\fL_1, z) \prod_{i=1}^N 
     {\fa_{\w \la^{(i)} - \w \mu^{(i, 1)}}(\w \alpha_i) \over
     \big (\w \la^{(i)} - \w \mu^{(i, 1)} \big )!}.
\end{eqnarray*}
By degree reason, $\Tr \, q^\fd \Gamma_+(-\fL_1, z) \fa_{\mu}(\beta) = 0$ if $|\mu| > 0$.
If $|\mu| = 0$, then we have $\Tr \, q^\fd \Gamma_+(-\fL_1, z) \fa_{\mu}(\beta) = 
\Tr \, q^\fd \fa_{\mu}(\beta)$. So $Q_2$ is equal to
\begin{eqnarray*}
\sum_{\sum_{i=1}^N(|\w \la^{(i)}| - |\w \mu^{(i, 1)}|) < 0} 
     \prod_{1 \le i \le N \atop n \ge 1} {z^{-n \w m^{(i,1)}_n} \over {\w m^{(i,1)}_n}!} 
     \cdot q^{\sum_{i=1}^N(|\w \mu^{(i, 1)}| - |\w \la^{(i)}|)} 
     \Tr \, q^\fd \Gamma_+(-\fL_1, z) \prod_{i=1}^N 
     {\fa_{\w \la^{(i)} - \w \mu^{(i, 1)}}(\w \alpha_i) \over
     \big (\w \la^{(i)} - \w \mu^{(i, 1)} \big )!}
\end{eqnarray*}
\begin{eqnarray*}
+ \sum_{\sum_{i=1}^N(|\w \la^{(i)}| - |\w \mu^{(i, 1)}|) = 0} 
     \prod_{1 \le i \le N \atop n \ge 1} {z^{-n \w m^{(i,1)}_n} \over {\w m^{(i,1)}_n}!} 
     \cdot \Tr \, q^\fd \prod_{i=1}^N 
     {\fa_{\w \la^{(i)} - \w \mu^{(i, 1)}}(\w \alpha_i) \over
     \big (\w \la^{(i)} - \w \mu^{(i, 1)} \big )!}.
\end{eqnarray*}

Repeating the process in the previous paragraph $t$ times, we conclude that 
$$
Q_2 = U(t) - V(t)
$$
where $U(t)$ is given by
\begin{eqnarray}    \label{Ut1}
\sum_{\sum_{i=1}^N(|\w \la^{(i)}| - \sum_{r=1}^t|\w \mu^{(i, r)}|) < 0} 
     \prod_{r=1}^t \left (\prod_{1 \le i \le N \atop n \ge 1} {z^{-n \w m^{(i,r)}_n} 
     \over {\w m^{(i,r)}_n}!} \cdot 
     q^{\sum_{i=1}^N(\sum_{\ell =1}^r |\w \mu^{(i, \ell)}| - |\w \la^{(i)}|)} \right ) \cdot
\end{eqnarray}
\begin{eqnarray}   \label{Ut2}
\cdot     \Tr \, q^\fd \Gamma_+(-\fL_1, z) \prod_{i=1}^N 
     {\fa_{\w \la^{(i)} - \sum_{r=1}^t \w \mu^{(i, r)}}(\w \alpha_i) \over
     \big (\w \la^{(i)} - \sum_{r=1}^t \w \mu^{(i, r)} \big )!}
\end{eqnarray}
with $\w \mu^{(i, r)} = \big (\cdots (-n)^{\w m^{(i,r)}_n} \cdots (-1)^{\w m^{(i,r)}_1} \big )
\in \Wcp_-$, and $V(t)$ is given by
\begin{eqnarray*}
\sum_{\sum_{i=1}^N(|\w \la^{(i)}| - \sum_{r=1}^t|\w \mu^{(i, r)}|) = 0} 
     \prod_{r=1}^t \left (\prod_{1 \le i \le N \atop n \ge 1} {z^{-n \w m^{(i,r)}_n} 
     \over {\w m^{(i,r)}_n}!} \cdot 
     q^{\sum_{i=1}^N(\sum_{\ell =1}^r |\w \mu^{(i, \ell)}| - |\w \la^{(i)}|)} \right ) \cdot
\end{eqnarray*}
\begin{eqnarray*}
\cdot     \Tr \, q^\fd \prod_{i=1}^N 
     {\fa_{\w \la^{(i)} - \sum_{r=1}^t \w \mu^{(i, r)}}(\w \alpha_i) \over
     \big (\w \la^{(i)} - \sum_{r=1}^t \w \mu^{(i, r)} \big )!}.
\end{eqnarray*}
Denote line \eqref{Ut1} by $\widetilde U(t)$.
Since $\sum_{i=1}^N(|\w \la^{(i)}| - \sum_{r=1}^t|\w \mu^{(i, r)}|) < 0$
and $|\w \mu^{(i, r)}| < 0$, $\widetilde U(t)$ is a polynomial in $q$ with coefficients 
being bounded in terms of $-\sum_{i=1}^N |\w \la^{(i)}|$. Moreover, $q^t|\widetilde U(t)$.
Line \eqref{Ut2} is bounded in terms of the generalized partitions $\w \la^{(i)}$.
Since $0 < q < 1$, $U(t) \to 0$ as $t \to +\infty$.
Letting $t \to +\infty$, we see that $Q_2$ equals
\begin{eqnarray*}
\sum_{\sum_{i=1}^N(|\w \la^{(i)}| - \sum_{t \ge 1} |\w \mu^{(i, t)}|) = 0} 
     \prod_{t \ge 1} \left (\prod_{1 \le i \le N \atop n \ge 1} {z^{-n \w m^{(i,t)}_n} 
     \over {\w m^{(i,t)}_n}!} \cdot 
     q^{\sum_{i=1}^N(\sum_{\ell =1}^t |\w \mu^{(i, \ell)}| - |\w \la^{(i)}|)} \right ) \cdot
\end{eqnarray*}
\begin{eqnarray*}
\cdot     \Tr \, q^\fd \prod_{i=1}^N 
     {\fa_{\w \la^{(i)} - \sum_{t \ge 1} \w \mu^{(i, t)}}(\w \alpha_i) \over
     \big (\w \la^{(i)} - \sum_{t \ge 1} \w \mu^{(i, t)} \big )!}.
\end{eqnarray*}
Replacing $q^{\sum_{i=1}^N(\sum_{\ell =1}^t |\w \mu^{(i, \ell)}| - |\w \la^{(i)}|)}$
by $q^{-\sum_{i=1}^N\sum_{\ell \ge t + 1} |\w \mu^{(i, \ell)}|}$ proves our lemma.
\end{proof}

\begin{lemma}  \label{Trqdjpx}
Let $\la^{(1)}, \ldots, \la^{(N)} \in \Wcp$ be generalized partitions, 
and $\alpha_1, \ldots, \alpha_N \in H^*(X)$ be homogeneous. Then, 
$\displaystyle{\Tr \, q^\fd \,\prod_{i=1}^N {\fa_{\la^{(i)}}(\alpha_i) \over \la^{(i)}!}}$ 
can be computed by an induction on $N$, and is a linear combination of expressions of the form:
\begin{eqnarray}     \label{Trqdjpx.0}
(q; q)_\infty^{-\chi(X)} \cdot {\rm Sign}(\pi) \cdot
\prod_{i=1}^u \left \langle e_X^{m_i}, \prod_{j \in \pi_i} \alpha_j \right \rangle \cdot
\prod_{i=1}^v {n_i^{k_i} q^{n_i} \over 1 - q^{n_i}}
\end{eqnarray}
where $0 \le v \le \sum_{i=1}^N \ell(\la^{(i)})/2$, $m_i \ge 0$, $n_i >0$,
all the integers involved and the partition $\{\pi_1, \ldots, \pi_u\}$ 
of $\{1, \ldots, N\}$ depend only on $\la^{(1)}, \ldots, \la^{(N)}$,  
and ${\rm Sign}(\pi)$ is the sign compensating the formal difference between
$\prod_{i=1}^u \prod_{j \in \pi_i} \alpha_j$ and $\alpha_1 \cdots \alpha_N$.
Moreover, the coefficients of this linear combination are independent of 
$q, \alpha_i, n_i, X$.
\end{lemma}
\begin{proof}
For simplicity, put $A_N = \displaystyle{\Tr \, q^\fd \,\prod_{i=1}^N {\fa_{\la^{(i)}}(\alpha_i) 
\over \la^{(i)}!}}$. Since $\fa_{\la^{(i)}}(\alpha_i)$ has conformal weight $|\la^{(i)}|$, 
$A_N = 0$ unless $\sum_{i=1}^N |\la^{(i)}| = 0$.
In the rest of the proof, we will assume $\sum_{i=1}^N |\la^{(i)}| = 0$. 
We will divide the proof into two cases.

\bigskip\noindent
{\bf Case 1:} $|\la^{(i)}| = 0$ for every $1 \le i \le N$.
Then, $\ell(\la^{(i)}) \ge 2$ for every $i$.
Since $\fa_{\la^{(i)}}(\alpha_i)$ has 
degree $2(\ell(\la^{(i)}) - 2) + |\alpha_i|$, $A_N = 0$ unless $\ell(\la^{(i)}) = 2$
and $|\alpha_i| = 0$ for all $1 \le i \le N$. Assume that $\ell(\la^{(i)}) = 2$
and $|\alpha_i| = 0$ for all $1 \le i \le N$. Then for every $1 \le i \le N$,
we have $\la^{(i)} = ((-n_i)n_i)$ for some $n_i > 0$. 
We further assume that $n_1 = \ldots = n_{r}$ for some $1 \le r \le N$
and $n_i \ne n_1$ if $r < i \le N$. Let $\alpha_1 = a 1_X$
and $\tau_{2*}1_X = \sum_j (-1)^{|\beta_j|} \beta_{j} \otimes \gamma_{j}$
with $\langle \beta_j, \gamma_{j'} \rangle = \delta_{j, j'}$. 
Then, $A_N$ is equal to 
\begin{eqnarray*}  
& &a \sum_j (-1)^{|\beta_j|} \Tr \, q^\fd \, \fa_{-n_1}(\beta_j) \, \fa_{n_1}(\gamma_j) 
     \prod_{i=2}^N \fa_{\la^{(i)}}(\alpha_i)   \\
&=&aq^{n_1} \sum_j (-1)^{|\beta_j|} \Tr \, \fa_{-n_1}(\beta_j) \, q^\fd \, \fa_{n_1}(\gamma_j) 
     \prod_{i=2}^N \fa_{\la^{(i)}}(\alpha_i)   \\
&=&aq^{n_1} \sum_j \Tr \, q^\fd \, \fa_{n_1}(\gamma_j) 
     \prod_{i=2}^N \fa_{\la^{(i)}}(\alpha_i) \cdot \fa_{-n_1}(\beta_j)  \\
&=&aq^{n_1} \sum_j \Tr \, q^\fd \, \fa_{n_1}(\gamma_j) \fa_{-n_1}(\beta_j)
     \prod_{i=2}^N \fa_{\la^{(i)}}(\alpha_i)   \\
& &\quad + aq^{n_1} \sum_j \sum_{i=2}^r \Tr \, q^\fd \, \fa_{n_1}(\gamma_j) 
     \prod_{k=2}^{i-1} \fa_{\la^{(k)}}(\alpha_k) 
     \cdot [\fa_{\la^{(i)}}(\alpha_i), \fa_{-n_1}(\beta_j)] \cdot    
     \prod_{k=i+1}^{N} \fa_{\la^{(k)}}(\alpha_k).
\end{eqnarray*}
By Lemma~\ref{tau_k_tau_{k-1}}~(i), $A_N$ is equal to the sum of the expressions
\begin{eqnarray}    \label{Trqdjpx.1}
\left \langle e_X, \alpha_1 \prod_{i =1}^{k_1} \alpha_{j_i} \right \rangle 
\cdot {(-n_1)^{k_1} q^{n_1} \over 1-q^{n_1}} \cdot
\Tr \, q^\fd \, \prod_{i \in \{2, \ldots, N\}-\{j_1, \ldots, j_{k_1} \}} \fa_{\la^{(i)}}(\alpha_i)
\end{eqnarray}
where $0 \le k_1 \le r-1$, $\{j_1, \ldots, j_{k_1} \} \subset \{2, \ldots, r \}$, 
every factor in $(-n_1)^{k_1}$ comes from a commutator of type 
$[\fa_{n_1}(\cdot), \fa_{-n_1}(\cdot)]$, and the coefficients 
of this linear combination depend only on $\la^{(1)}, \ldots, \la^{(N)}$. 
In particular, we have 
\begin{eqnarray}    \label{Trqdjpx.2}
A_1 = \Tr \, q^\fd \, \fa_{\la^{(1)}}(\alpha_1) 
= (q; q)_\infty^{-\chi(X)} \cdot \langle e_X, \alpha_1 \rangle 
  \cdot {(-n_1) q^{n_1} \over 1-q^{n_1}}.
\end{eqnarray}
Combining with \eqref{Trqdjpx.1}, we see that our lemma holds in this case.

\bigskip\noindent
{\bf Case 2:} $\sum_{i=1}^N |\la^{(i)}| = 0$ but $|\la^{(i_0)}| \ne 0$ 
for some $i_0$. Then, $N \ge 2$, and we may assume that $|\la^{(i_0)}| < 0$. 
To simplify the expressions, we further assume that 
every $\alpha_i$ has an even degree. Note that $A_N$ can be rewritten as
\begin{eqnarray*}  
& &\Tr \, q^\fd {\fa_{\la^{(i_0)}}(\alpha_{i_0}) \over \la^{({i_0})}!} 
     \prod_{1 \le i \le N, i \ne i_0} {\fa_{\la^{(i)}}(\alpha_i) \over \la^{(i)}!}      \\
&+&\sum_{r = 1}^{i_0 -1} \Tr \, q^\fd \prod_{i=1}^{r-1} {\fa_{\la^{(i)}}(\alpha_i) \over \la^{(i)}!}
    \cdot \left [{\fa_{\la^{(r)}}(\alpha_r) \over \la^{(r)}!},
      {\fa_{\la^{(i_0)}}(\alpha_{i_0}) \over \la^{({i_0})}!} \right ] 
    \cdot \prod_{r+1 \le i \le N, i \ne i_0} {\fa_{\la^{(i)}}(\alpha_i) \over \la^{(i)}!}.
\end{eqnarray*}
Since $q^\fd \fa_{\la^{(i_0)}}(\alpha_{i_0}) 
= q^{-|\la^{(i_0)}|} \fa_{\la^{(i_0)}}(\alpha_{i_0})  q^\fd$, we see that $A_N$ is equal to
\begin{eqnarray*}  
& &q^{-|\la^{(i_0)}|} \, \Tr \, {\fa_{\la^{(i_0)}}(\alpha_{i_0}) \over \la^{({i_0})}!} q^\fd 
     \prod_{1 \le i \le N, i \ne i_0} {\fa_{\la^{(i)}}(\alpha_i) \over \la^{(i)}!}      \\
& &+ \sum_{r = 1}^{i_0 -1} \Tr \, q^\fd \prod_{i=1}^{r-1} {\fa_{\la^{(i)}}(\alpha_i) \over \la^{(i)}!}
    \cdot \left [{\fa_{\la^{(r)}}(\alpha_r) \over \la^{(r)}!},
      {\fa_{\la^{(i_0)}}(\alpha_{i_0}) \over \la^{({i_0})}!} \right ] 
    \cdot \prod_{r+1 \le i \le N, i \ne i_0} {\fa_{\la^{(i)}}(\alpha_i) \over \la^{(i)}!}  \\
&=&q^{-|\la^{(i_0)}|} \, \Tr \, q^\fd 
     \prod_{1 \le i \le N, i \ne i_0} {\fa_{\la^{(i)}}(\alpha_i) \over \la^{(i)}!} \cdot
     {\fa_{\la^{(i_0)}}(\alpha_{i_0}) \over \la^{({i_0})}!}      \\
& &+ \sum_{r = 1}^{i_0 -1} \Tr \, q^\fd \prod_{i=1}^{r-1} {\fa_{\la^{(i)}}(\alpha_i) \over \la^{(i)}!}
    \cdot \left [{\fa_{\la^{(r)}}(\alpha_r) \over \la^{(r)}!},
      {\fa_{\la^{(i_0)}}(\alpha_{i_0}) \over \la^{({i_0})}!} \right ] 
    \cdot \prod_{r+1 \le i \le N, i \ne i_0} {\fa_{\la^{(i)}}(\alpha_i) \over \la^{(i)}!}.
\end{eqnarray*}
Note that $\Tr \, q^\fd \prod_{1 \le i \le N, i \ne i_0} 
{\fa_{\la^{(i)}}(\alpha_i) \over \la^{(i)}!} \cdot
{\fa_{\la^{(i_0)}}(\alpha_{i_0}) \over \la^{({i_0})}!}$ is equal to
$$
A_N + \sum_{r = i_0+1}^N \Tr \, q^\fd 
\prod_{1 \le i \le r-1, i \ne i_0} {\fa_{\la^{(i)}}(\alpha_i) \over \la^{(i)}!} 
\cdot \left [{\fa_{\la^{(r)}}(\alpha_r) \over \la^{(r)}!},
      {\fa_{\la^{(i_0)}}(\alpha_{i_0}) \over \la^{({i_0})}!} \right ] \cdot
\prod_{i = r+1}^N {\fa_{\la^{(i)}}(\alpha_i) \over \la^{(i)}!}.
$$
Therefore, we conclude that $(1 - q^{-|\la^{(i_0)}|})A_N$ is equal to
$$
q^{-|\la^{(i_0)}|} \, \sum_{r = i_0+1}^N \Tr \, q^\fd 
\prod_{1 \le i \le r-1, i \ne i_0} {\fa_{\la^{(i)}}(\alpha_i) \over \la^{(i)}!} 
\cdot \left [{\fa_{\la^{(r)}}(\alpha_r) \over \la^{(r)}!},
      {\fa_{\la^{(i_0)}}(\alpha_{i_0}) \over \la^{({i_0})}!} \right ] \cdot
\prod_{i = r+1}^N {\fa_{\la^{(i)}}(\alpha_i) \over \la^{(i)}!}
$$
$$
+ \sum_{r = 1}^{i_0 -1} \Tr \, q^\fd \prod_{i=1}^{r-1} {\fa_{\la^{(i)}}(\alpha_i) \over \la^{(i)}!}
    \cdot \left [{\fa_{\la^{(r)}}(\alpha_r) \over \la^{(r)}!},
      {\fa_{\la^{(i_0)}}(\alpha_{i_0}) \over \la^{({i_0})}!} \right ] 
    \cdot \prod_{r+1 \le i \le N, i \ne i_0} {\fa_{\la^{(i)}}(\alpha_i) \over \la^{(i)}!}.
$$
Put $n_0 = -|\la^{(i_0)}| > 0$. It follows that $A_N$ is equal to 
$$
{q^{n_0} \over 1 - q^{n_0}} \, \sum_{r = i_0+1}^N \Tr \, q^\fd 
\prod_{1 \le i \le r-1, i \ne i_0} {\fa_{\la^{(i)}}(\alpha_i) \over \la^{(i)}!} 
\cdot \left [{\fa_{\la^{(r)}}(\alpha_r) \over \la^{(r)}!},
      {\fa_{\la^{(i_0)}}(\alpha_{i_0}) \over \la^{({i_0})}!} \right ] \cdot
\prod_{i = r+1}^N {\fa_{\la^{(i)}}(\alpha_i) \over \la^{(i)}!}
$$
\begin{eqnarray}    \label{Trqdjpx.3}
+ {1 \over 1 - q^{n_0}} \, 
\sum_{r = 1}^{i_0 -1} \Tr \, q^\fd \prod_{i=1}^{r-1} {\fa_{\la^{(i)}}(\alpha_i) \over \la^{(i)}!}
    \cdot \left [{\fa_{\la^{(r)}}(\alpha_r) \over \la^{(r)}!},
      {\fa_{\la^{(i_0)}}(\alpha_{i_0}) \over \la^{({i_0})}!} \right ] 
    \cdot \prod_{r+1 \le i \le N, i \ne i_0} {\fa_{\la^{(i)}}(\alpha_i) \over \la^{(i)}!}. \quad
\end{eqnarray}
By Lemma~\ref{tau_k_tau_{k-1}}~(i) and (ii), our lemma holds in this case as well.
\end{proof}

The following theorem provides the structure of the trace 
$$
\Tr \, q^\fd \, W(\fL_1, z) \, 
\prod_{i=1}^N {\fa_{\la^{(i)}}(\alpha_i) \over \la^{(i)}!}.
$$

\begin{theorem}   \label{ThmJJkAlpha}
For $1 \le i \le N$, let $\la^{(i)} = \big (\cdots (-n)^{\w m^{(i)}_{n}} \cdots (-1)^{\w m^{(i)}_1}
1^{m^{(i)}_1} \cdots n^{m^{(i)}_{n}} \cdots )$ and $\alpha_i \in H^*(X)$ be homogeneous. 
Then, $\displaystyle{\Tr \, q^\fd \, W(\fL_1, z) \, 
\prod_{i=1}^N {\fa_{\la^{(i)}}(\alpha_i) \over \la^{(i)}!}}$ is equal to
$$   
z^{\sum_{i=1}^N |\la^{(i)}|} \cdot (q; q)_\infty^{-\chi(X)} \cdot \prod_{i=1}^N 
\big \langle (1_X - K_X)^{\sum_{n \ge 1} m^{(i)}_{n}}, \alpha_i \big \rangle \cdot
$$
$$
\cdot \prod_{1 \le i \le N, n \ge 1} \left ( {(-1)^{m^{(i)}_{n}} \over m^{(i)}_{n}!} 
    {q^{n m^{(i)}_{n}} \over (1-q^n)^{m^{(i)}_{n}}} {1 \over \w m^{(i)}_{n}!} 
    {1 \over (1-q^n)^{\w m^{(i)}_{n}}} \right ) + \widetilde W,
$$
and the lower weight term $\widetilde W$ is a linear combination of expressions of the form:
\begin{eqnarray}    \label{ThmJJkAlpha.0}
z^{\sum_{i=1}^N |\la^{(i)}|} \cdot (q; q)_\infty^{-\chi(X)} \cdot {\rm Sign}(\pi) \cdot
\prod_{i=1}^u \left \langle K_X^{r_i}e_X^{r_i'}, \prod_{j \in \pi_i} 
\alpha_j \right \rangle \cdot
\prod_{i=1}^v {q^{n_i w_ip_i} \over (1-q^{n_i})^{w_i}}  
\end{eqnarray} 
where $\sum_{i=1}^v w_i < \sum_{i=1}^N \ell(\la^{(i)})$, 
the integers $u, v$, $r_i, r_i' \ge 0, n_i > 0, 
w_i >0, p_i \in \{0, 1\}$ and the partition $\pi = \{\pi_1, \ldots, \pi_u\}$
of $\{1, \ldots, N\}$ depend only on the generalized partitions $\la^{(1)}, \ldots, \la^{(N)}$, and
${\rm Sign}(\pi)$ is the sign compensating the formal difference between
$\prod_{i=1}^u \prod_{j \in \pi_i} \alpha_j$ and $\alpha_1 \cdots \alpha_N$.
Moreover, the coefficients of this linear combination are independent of 
$q, \alpha_1, \ldots, \alpha_N$ and $X$.
\end{theorem}
\begin{proof}
For simplicity, put $\Tr_\la = \displaystyle{\Tr \, q^\fd \, W(\fL_1, z) \, 
\prod_{i=1}^N {\fa_{\la^{(i)}}(\alpha_i) \over \la^{(i)}!}}$.
Combining Lemma~\ref{LemmaTrace} and Lemma~\ref{20150806957}, 
we conclude that $\Tr_\la$ is equal to
$$
\sum_{\sum_{i=1}^N (|\la^{(i)}| - \sum_{s \ge 1} |\mu^{(i, s)}| 
  - \sum_{t \ge 1} |\w \mu^{(i, t)}|) = 0
\atop \mu^{(i, s)} \in \Wcp_+, \, \w \mu^{(i, t)} \in \Wcp_-} \,
\prod_{1 \le i \le N \atop s, n \ge 1} {(-(zq^s)^n)^{m^{(i, s)}_{n}} \over m^{(i, s)}_{n}!} \cdot
\prod_{1 \le i \le N \atop t, n \ge 1} {(z^{-1}q^{t-1})^{n \w m^{(i,t)}_n} 
     \over {\w m^{(i,t)}_n}!}  \cdot 
$$
\begin{eqnarray*} 
\cdot     \Tr \, q^\fd \prod_{i=1}^N 
     {\fa_{\la^{(i)}-\sum_{s \ge 1} \mu^{(i, s)} - \sum_{t \ge 1} \w \mu^{(i, t)}}
\big ((1_X - K_X)^{\sum_{s, n \ge 1} m^{(i, s)}_{n}}\alpha_i \big ) \over
     \big (\la^{(i)}-\sum_{s \ge 1} \mu^{(i, s)} - \sum_{t \ge 1} \w \mu^{(i, t)}\big )!}
\end{eqnarray*}
where $\mu^{(i, s)} = \big (1^{m^{(i, s)}_{1}}\cdots n^{m^{(i, s)}_{n}}\cdots \big )$
and $\w \mu^{(i, t)} = \big (\cdots (-n)^{\w m^{(i, t)}_{n}}\cdots (-1)^{\w m^{(i, t)}_{1}}\big )$.
The sum of all the exponents of $z$ is $\sum_{i=1}^N |\la^{(i)}|$. So $\Tr_\la$ is equal to
$$
z^{\sum_{i=1}^N |\la^{(i)}|} \cdot
\sum_{\sum_{i=1}^N (|\la^{(i)}| - \sum_{s \ge 1} |\mu^{(i, s)}| - \sum_{t \ge 1} |\w \mu^{(i, t)}|) = 0
\atop \mu^{(i, s)} \in \Wcp_+, \, \w \mu^{(i, t)} \in \Wcp_-} \,
\prod_{1 \le i \le N \atop s, n \ge 1} {(-q^{sn})^{m^{(i, s)}_{n}} \over m^{(i, s)}_{n}!} \cdot
\prod_{1 \le i \le N \atop t, n \ge 1} {q^{(t-1)n \w m^{(i,t)}_n} 
     \over {\w m^{(i,t)}_n}!}  \cdot 
$$
\begin{eqnarray}   \label{ThmJJkAlpha.1}
\cdot     \Tr \, q^\fd \prod_{i=1}^N 
     {\fa_{\la^{(i)}-\sum_{s \ge 1} \mu^{(i, s)} - \sum_{t \ge 1} \w \mu^{(i, t)}}
\big ((1_X - K_X)^{\sum_{s, n \ge 1} m^{(i, s)}_{n}}\alpha_i \big ) \over
     \big (\la^{(i)}-\sum_{s \ge 1} \mu^{(i, s)} - \sum_{t \ge 1} \w \mu^{(i, t)}\big )!}.
\end{eqnarray}
By our convention, $\sum_{s \ge 1} \mu^{(i, s)} + \sum_{t \ge 1} \w \mu^{(i, t)}
\le \la^{(i)}$ for every $1 \le i \le N$. We now divide the rest of the proof into 
Case A and Case B.

\bigskip\noindent
{\bf Case A:}
$\sum_{s \ge 1} \mu^{(i, s)} + \sum_{t \ge 1} \w \mu^{(i, t)}
= \la^{(i)}$ for every $1 \le i \le N$. Then line \eqref{ThmJJkAlpha.1} is 
$$     
\Tr \, q^\fd \cdot \prod_{i=1}^N 
\big \langle (1_X - K_X)^{\sum_{s, n \ge 1} m^{(i, s)}_{n}}, \alpha_i \big \rangle
= (q; q)_\infty^{-\chi(X)} \prod_{i=1}^N 
\big \langle (1_X - K_X)^{\sum_{s, n \ge 1} m^{(i, s)}_{n}}, \alpha_i \big \rangle.
$$
Therefore, the contribution $C_1$ of this case to $\Tr_\la$ is equal to
$$   
z^{\sum_{i=1}^N |\la^{(i)}|} \cdot (q; q)_\infty^{-\chi(X)} \prod_{i=1}^N 
\big \langle (1_X - K_X)^{\sum_{n \ge 1} m^{(i)}_{n}}, \alpha_i \big \rangle 
\cdot 
$$
$$
\cdot \sum_{\sum_{s \ge 1} m^{(i, s)}_{n} = m^{(i)}_{n} \atop 1 \le i \le N,\, n \ge 1} \,
\prod_{1 \le i \le N \atop s, n \ge 1} {(-q^{sn})^{m^{(i, s)}_{n}} \over m^{(i, s)}_{n}!} 
\cdot \sum_{\sum_{t \ge 1} \w m^{(i, t)}_{n} = \w m^{(i)}_{n} \atop 1 \le i \le N,\, n \ge 1} \,
\prod_{1 \le i \le N \atop t, n \ge 1} {q^{(t-1)n \w m^{(i,t)}_n} 
     \over {\w m^{(i,t)}_n}!}. 
$$
Rewrite $q^{sn}$ as $q^{(s-1)n}q^n$. Then the contribution $C_1$ is equal to
$$   
z^{\sum_{i=1}^N |\la^{(i)}|} \cdot (q; q)_\infty^{-\chi(X)} \prod_{i=1}^N 
\big \langle (1_X - K_X)^{\sum_{n \ge 1} m^{(i)}_{n}}, \alpha_i \big \rangle 
\cdot \prod_{1 \le i \le N, n \ge 1} (-q^n)^{m^{(i)}_{n}} \cdot
$$
$$
\cdot \sum_{\sum_{s \ge 1} m^{(i, s)}_{n} = m^{(i)}_{n} \atop 1 \le i \le N,\, n \ge 1} \,
\prod_{1 \le i \le N \atop s, n \ge 1} {q^{(s-1)nm^{(i, s)}_{n}} \over m^{(i, s)}_{n}!} 
\cdot \sum_{\sum_{t \ge 1} \w m^{(i, t)}_{n} = \w m^{(i)}_{n} \atop 1 \le i \le N,\, n \ge 1} \,
\prod_{1 \le i \le N \atop t, n \ge 1} {q^{(t-1)n \w m^{(i,t)}_n} 
     \over {\w m^{(i,t)}_n}!}.
$$
Since $\displaystyle{\sum_{\sum_{s \ge 1} i_{s,n} = i_n, \, n \ge 1} \,
\prod_{s, n \ge 1} {(q^{(s-1)n})^{i_{s,n}} \over i_{s,n}!} 
= \prod_{n \ge 1} \left ( {1 \over i_{n}!} {1 \over (1-q^n)^{i_n}} \right )}$,
$C_1$ is equal to
$$   
z^{\sum_{i=1}^N |\la^{(i)}|} \cdot (q; q)_\infty^{-\chi(X)} \prod_{i=1}^N 
\big \langle (1_X - K_X)^{\sum_{n \ge 1} m^{(i)}_{n}}, \alpha_i \big \rangle \cdot
$$
\begin{eqnarray}     \label{ThmJJkAlpha.2}
\cdot \prod_{1 \le i \le N, n \ge 1} \left ( {(-1)^{m^{(i)}_{n}} \over m^{(i)}_{n}!} 
    {q^{n m^{(i)}_{n}} \over (1-q^n)^{m^{(i)}_{n}}} \right )
\cdot \prod_{1 \le i \le N, n \ge 1} \left ( {1 \over \w m^{(i)}_{n}!} 
    {1 \over (1-q^n)^{\w m^{(i)}_{n}}} \right ).
\end{eqnarray}

\bigskip\noindent
{\bf Case B:} $\sum_{s \ge 1} \mu^{(i, s)} + \sum_{t \ge 1} \w \mu^{(i, t)}
< \la^{(i)}$ for some $1 \le i \le N$. Without loss of generality, we may 
assume that $\sum_{s \ge 1} \mu^{(i, s)} + \sum_{t \ge 1} \w \mu^{(i, t)}
= \la^{(i)}$ for every $1 \le i \le N_1$ where $N_1< N$,
and $\sum_{s \ge 1} \mu^{(i, s)} + \sum_{t \ge 1} \w \mu^{(i, t)}
< \la^{(i)}$ for every $N_1 + 1 \le i \le N$. For $N_1+1 \le i \le N$, put
$\sum_{s \ge 1} \mu^{(i, s)} + \sum_{t \ge 1} \w \mu^{(i, t)} = \w \la^{(i)}
= \big ( \cdots (-n)^{\w p^{(i)}_n} \cdots (-1)^{\w p^{(i)}_1} 1^{p^{(i)}_1}
\cdots n^{p^{(i)}_n} \cdots \big )$.
An argument similar to that in the previous paragraph shows that 
for the fixed generalized partitions $\w \la^{(i)}$ with $N_1+1 \le i \le N$, 
the contribution $C_2$ of this case to $\Tr_\la$ is equal to
$$   
z^{\sum_{i=1}^N |\la^{(i)}|} \cdot \prod_{i=1}^{N_1} 
\big \langle (1_X - K_X)^{\sum_{n \ge 1} m^{(i)}_{n}}, \alpha_i \big \rangle \cdot
\prod_{1 \le i \le N_1 \atop n \ge 1} \left ( {(-1)^{m^{(i)}_{n}} \over m^{(i)}_{n}!} 
    {q^{n m^{(i)}_{n}} \over (1-q^n)^{m^{(i)}_{n}}}  {1 \over \w m^{(i)}_{n}!} 
    {1 \over (1-q^n)^{\w m^{(i)}_{n}}} \right )
$$
$$
\cdot 
\prod_{N_1+1 \le i \le N \atop n \ge 1} \left ( {(-1)^{p^{(i)}_{n}} \over p^{(i)}_{n}!} 
    {q^{n p^{(i)}_{n}} \over (1-q^n)^{p^{(i)}_{n}}}  {1 \over \w p^{(i)}_{n}!} 
    {1 \over (1-q^n)^{\w p^{(i)}_{n}}} \right )
$$
\begin{eqnarray}     \label{ThmJJkAlpha.3}
\cdot     \Tr \, q^\fd \prod_{i=N_1+1}^N 
     {\fa_{\la^{(i)} - \w \la^{(i)}}
\big ((1_X - K_X)^{\sum_{n \ge 1} p^{(i)}_{n}}\alpha_i \big ) \over
     \big (\la^{(i)}-\w \la^{(i)} \big )!}.
\end{eqnarray}
By Lemma~\ref{Trqdjpx}, $C_2$ is a linear combination of expressions of the form:
$$   
z^{\sum_{i=1}^N |\la^{(i)}|} \cdot \prod_{i=1}^{N_1} 
\big \langle (1_X - K_X)^{\sum_{n \ge 1} m^{(i)}_{n}}, \alpha_i \big \rangle \cdot
\prod_{1 \le i \le N_1 \atop n \ge 1} \left ( {(-1)^{m^{(i)}_{n}} \over m^{(i)}_{n}!} 
    {q^{n m^{(i)}_{n}} \over (1-q^n)^{m^{(i)}_{n}}}  {1 \over \w m^{(i)}_{n}!} 
    {1 \over (1-q^n)^{\w m^{(i)}_{n}}} \right )
$$
\begin{eqnarray*}
\cdot 
\prod_{N_1+1 \le i \le N \atop n \ge 1} \left ( {(-1)^{p^{(i)}_{n}} \over p^{(i)}_{n}!} 
    {q^{n p^{(i)}_{n}} \over (1-q^n)^{p^{(i)}_{n}}}  {1 \over \w p^{(i)}_{n}!} 
    {1 \over (1-q^n)^{\w p^{(i)}_{n}}} \right ) \cdot     
\end{eqnarray*}
\begin{eqnarray*}     
(q; q)_\infty^{-\chi(X)} \cdot {\rm Sign}(\pi) \cdot
\prod_{i=1}^u \left \langle e_X^{m_i}, \prod_{j \in \pi_i} 
\big ((1_X - K_X)^{\sum_{n \ge 1} p^{(j)}_{n}}\alpha_j \big ) \right \rangle \cdot
\prod_{i=1}^v {q^{n_i} \over 1 - q^{n_i}}
\end{eqnarray*}
where $v < \sum_{i=N_1+1}^N \ell(\la^{(i)} - \w \la^{(i)})$, $n_i > 0$, $m_i \ge 0$,
$\{\pi_1, \ldots, \pi_u\}$ is a partition of $\{N_1+1, \ldots, N\}$,
and ${\rm Sign}(\pi)$ compensates the formal difference between
$\prod_{i=1}^u \prod_{j \in \pi_i} \alpha_j$ and $\alpha_{N_1+1} \cdots \alpha_N$.
The coefficients of this linear combination are independent of 
$q, \alpha_1, \ldots, \alpha_N$ and $X$, and depend only on 
the partitions $\la^{(i)} - \w \la^{(i)}$. Note that 
for nonnegative integers $a$ and $b$, the pairing 
$\langle e_X^a, (1_X -K_X)^b \beta \rangle = \langle e_X^a(1_X -K_X)^b, \beta \rangle$
is a linear combination of $\langle e_X^a K_X^c, \beta \rangle, 0 \le c \le b$.
In addition, we have 
$$
\sum_{1 \le i \le N_1, n \ge 1} \big ( m^{(i)}_{n} + \w m^{(i)}_{n}\big )
+ \sum_{N_1+1 \le i \le N, n \ge 1} \big ( p^{(i)}_{n} + \w p^{(i)}_{n}\big ) + v
< \sum_{i=1}^N \ell(\la^{(i)})
$$
regarding the weights in $C_2$.
It follows that $C_2$ is a linear combination of the expressions \eqref{ThmJJkAlpha.0}.
Combining with \eqref{ThmJJkAlpha.2} completes the proof of our theorem.
\end{proof}

\begin{remark}   \label{Rmk-914}
When $N = 1$, we can work out the lower weight term $\widetilde W$ in Theorem~\ref{ThmJJkAlpha}
by examining its proof more carefully and by using \eqref{Trqdjpx.2}. To state the result,
let $\la = (\cdots (-n)^{\w m_n} \cdots (-1)^{\w m_1} 1^{m_1}\cdots n^{m_n} \cdots) \in \Wcp$. 
For $n_1 \ge 1$ with $m_{n_1} \cdot \w m_{n_1} \ge 1$, define
$m_{n_1}(n_1) = m_{n_1}-1$, $\w m_{n_1}(n_1) = \w m_{n_1}-1$,
and $m_n(n_1) = m_n$ and $\w m_n(n_1) = \w m_n$ if $n \ne n_1$. Then, 
$\displaystyle{\Tr \, q^\fd \, W(\fL_1, z) \, {\fa_\la(\alpha) \over \la!}}$ 
is equal to the sum
\begin{eqnarray*}  
& &z^{|\la|} \cdot (q; q)_\infty^{-\chi(X)} \cdot 
   \langle (1_X - K_X)^{\sum_{n \ge 1} m_{n}}, \alpha \rangle  \cdot  \\
& &\quad \cdot \prod_{n \ge 1} \left ( {(-1)^{m_{n}} \over m_{n}!} {q^{n m_n} \over (1-q^n)^{m_n}} 
   {1 \over \w m_{n}!} {1 \over (1-q^n)^{\w m_n}} \right )  \\
&+&z^{|\la|} \cdot (q; q)_\infty^{-\chi(X)} \cdot \langle e_X, \alpha \rangle 
   \cdot \sum_{n_1 \ge 1 \, \text{\rm with } m_{n_1} \cdot \w m_{n_1} \ge 1} 
   {n_1 q^{n_1} \over 1-q^{n_1}} \cdot   \\
& &\quad \cdot \prod_{n \ge 1} \left ( {(-1)^{m_{n}(n_1)} \over m_{n}(n_1)!} 
   {q^{n m_n(n_1)} \over (1-q^n)^{m_n(n_1)}} 
   {1 \over \w m_{n}(n_1)!} {1 \over (1-q^n)^{\w m_n(n_1)}} \right ).
\end{eqnarray*}
\end{remark}

The next lemma is used to organize the leading term in Theorem~\ref{ThmJJkAlpha}.

\begin{lemma}   \label{ThetaAlphaK}
For $\alpha \in H^*(X)$ and $k \ge 0$, define $\Theta^\alpha_k(q)$ to be 
\begin{eqnarray}   \label{ThetaAlphaK.01}
-\sum_{\ell(\la) = k +2, |\la|=0}
  \big \langle (1_X - K_X)^{\sum_{n \ge 1} i_{n}}, \alpha \big \rangle \cdot
  \prod_{n \ge 1} \left ( {(-1)^{i_{n}} \over i_{n}!} 
  {q^{n i_{n}} \over (1-q^n)^{i_{n}}} {1 \over \w i_{n}!} 
  {1 \over (1-q^n)^{\w i_{n}}} \right )   \quad
\end{eqnarray}
where $\la = (\cdots (-n)^{\w i_{n}} \cdots (-1)^{\w i_1} 1^{i_1} \cdots n^{i_{n}} \cdots )$. 
Then, $\Theta^\alpha_k(q) = \Coe_{z^0} \Theta^\alpha_k(q, z)$ 
which denotes the coefficient of $z^0$ in $\Theta^\alpha_k(q, z)$ defined by
$$    
-\sum_{a, s_1, \ldots, s_a, b, t_1, \ldots, t_b \ge 1 
\atop \sum_{i=1}^a s_i + \sum_{j = 1}^b t_j = k+2} 
\big \langle (1_X - K_X)^{\sum_{i = 1}^a s_i}, \alpha \big \rangle
\prod_{i=1}^a {(-1)^{s_i} \over s_i!} \cdot \prod_{j=1}^b {1 \over t_j!} 
$$
\begin{eqnarray}   \label{ThetaAlphaK.02}
\cdot \sum_{n_1 > \cdots > n_a} \prod_{i= 1}^a {(q z)^{n_i s_i} \over (1-q^{n_i})^{s_i}}
\cdot \sum_{m_1 > \cdots > m_b} \prod_{j= 1}^b {z^{-m_jt_j} \over (1-q^{m_j})^{t_j}}.
\end{eqnarray} 
\end{lemma}
\begin{proof}
Put $A = \big \langle (1_X - K_X)^{\sum_{n \ge 1} i_{n}}, \alpha \big \rangle$
which implicitly depends on $\sum_{n \ge 1} i_{n}$. Rewrite $|\la|$ and $\ell(\la)$ in terms of 
the integers $i_n$ and $\w i_n$. Then, $\Theta^\alpha_k(q)$ is equal to
\begin{eqnarray}     \label{ThetaAlphaK.1}
-\sum_{\sum_{n \ge 1} i_n + \sum_{n \ge 1} \w i_n = k +2
  \atop \sum_{n \ge 1} n i_{n} = \sum_{n \ge 1} n \w i_{n} > 0} A \,
  \prod_{n \ge 1} \left ( {(-1)^{i_{n}} \over i_{n}!} 
  {q^{n i_{n}} \over (1-q^n)^{i_{n}}} {1 \over \w i_{n}!} 
  {1 \over (1-q^n)^{\w i_{n}}} \right ).
\end{eqnarray}
Denote the {\it positive} integers in the {\it ordered} list $\{i_1, \ldots, i_n, \ldots \}$
by $s_a, \ldots, s_1$ respectively (e.g., if the ordered list 
$\{i_1, \ldots, i_n, \ldots \}$ is $\{2, 0, 5, 4, 0, \ldots \}$, then $a=3$ with 
$s_3 = 2, s_2 =5, s_1 = 4$). We have $a \ge 1$. Similarly, 
denote the {\it positive} integers in the {\it ordered} list 
$\{ \w i_1, \ldots, \w i_n, \ldots \}$ by $t_b, \ldots, t_1$ respectively. Then $b \ge 1$.
Since $\sum_{n \ge 1} i_{n} = \sum_{i=1}^a s_i$, we get
$A = \big \langle (1_X - K_X)^{\sum_{i = 1}^a s_i}, \alpha \big \rangle$.
Rewriting \eqref{ThetaAlphaK.1} in terms of $s_a, \ldots, s_1$ and $t_b, \ldots, t_1$, 
we see that $\Theta^\alpha_k(q) = \Coe_{z^0} \Theta^\alpha_k(q, z)$.
\end{proof}

We remark that the multiple $q$-zeta value $\Theta^\alpha_k(q, z)$ has weight $(k+2)$.

\begin{theorem}   \label{ThmFk1NAlpha1N}
For $1 \le i \le N$, let $k_i \ge 0$ and $\alpha_i \in H^*(X)$ be homogeneous. Then, 
\begin{eqnarray}   \label{ThmFk1NAlpha1N.00}
F^{\alpha_1, \ldots, \alpha_N}_{k_1, \ldots, k_N}(q) 
= (q; q)_\infty^{-\chi(X)} \cdot \Coe_{z_1^0 \cdots z_N^0} 
\left (\prod_{i=1}^N \Theta^{\alpha_i}_{k_i}(q, z_i) \right ) + W_1,
\end{eqnarray}
and the lower weight term $W_1$ is an infinite linear combination of the expressions:
\begin{eqnarray}   \label{ThmFk1NAlpha1N.0}   
(q; q)_\infty^{-\chi(X)} \cdot {\rm Sign}(\pi) \cdot
\prod_{i=1}^u \left \langle K_X^{r_i}e_X^{r_i'}, \prod_{j \in \pi_i} 
\alpha_j \right \rangle \cdot
\prod_{i=1}^v {q^{n_i w_ip_i} \over (1-q^{n_i})^{w_i}}
\end{eqnarray} 
where $\sum_{i=1}^v w_i < \sum_{i=1}^N (k_i + 2)$, 
and the integers $u, v, r_i, r_i' \ge 0, 
n_i > 0, w_i >0, p_i \in \{0, 1\}$ and the partition $\pi = \{\pi_1, \ldots, \pi_u\}$
of $\{1, \ldots, N\}$ depend only on the integers $k_i$. 
Moreover, the coefficients of this linear combination are independent of 
$q, \alpha_i, X$.
\end{theorem}
\begin{proof}
By Lemma~\ref{FtoW}, $\displaystyle{F^{\alpha_1, \ldots, \alpha_N}_{k_1, \ldots, k_N}(q) 
= \Tr \, q^\fd \, W(\fL_1, z) \, \prod_{i=1}^N \fG_{k_i}(\alpha_i)}$.
Combining with Theorem~\ref{char_th} and Theorem~\ref{ThmJJkAlpha}, we conclude that
\begin{eqnarray}     \label{ThmFk1NAlpha1N.1}
F^{\alpha_1, \ldots, \alpha_N}_{k_1, \ldots, k_N}(q)
= \widetilde F^{\alpha_1, \ldots, \alpha_N}_{k_1, \ldots, k_N}(q) + W_{1,1}
\end{eqnarray}    
where $W_{1,1}$ is an infinite linear combination of 
the expressions \eqref{ThmFk1NAlpha1N.0}, and  
\begin{eqnarray}     \label{DefWideF}
\widetilde F^{\alpha_1, \ldots, \alpha_N}_{k_1, \ldots, k_N}(q) 
:= (-1)^N \cdot \sum_{\ell(\la^{(i)}) = k_i +2, |\la^{(i)}|=0 \atop 1 \le i \le N} 
\Tr \, q^\fd \, W(\fL_1, z) \, \prod_{i=1}^N
   {\mathfrak a_{\la^{(i)}}(\alpha_i) \over \la^{(i)}!}.
\end{eqnarray} 
Applying Theorem~\ref{ThmJJkAlpha} again, 
we see that $\widetilde F^{\alpha_1, \ldots, \alpha_N}_{k_1, \ldots, k_N}(q)$ is equal to
$$
(-1)^N (q; q)_\infty^{-\chi(X)} \cdot 
\sum_{\ell(\la^{(i)}) = k_i +2, |\la^{(i)}|=0 \atop 1 \le i \le N} 
\prod_{i=1}^N \big \langle (1_X - K_X)^{\sum_{n \ge 1} m^{(i)}_{n}}, \alpha_i \big \rangle \cdot
$$
\begin{eqnarray}     \label{ThmFk1NAlpha1N.2}
\cdot \prod_{1 \le i \le N, n \ge 1} \left ( {(-1)^{m^{(i)}_{n}} \over m^{(i)}_{n}!} 
    {q^{n m^{(i)}_{n}} \over (1-q^n)^{m^{(i)}_{n}}} {1 \over \w m^{(i)}_{n}!} 
    {1 \over (1-q^n)^{\w m^{(i)}_{n}}} \right ) + W_{1, 2}
\end{eqnarray}    
where the lower weight term $W_{1, 2}$ is an infinite linear combination of 
the expressions \eqref{ThmFk1NAlpha1N.0}, and we have put
$\la^{(i)} = \big (\cdots (-n)^{\w m^{(i)}_{n}} \cdots (-1)^{\w m^{(i)}_1}
1^{m^{(i)}_1} \cdots n^{m^{(i)}_{n}} \cdots )$. So
\begin{eqnarray}     \label{ThmFk1NAlpha1N.3}
   \widetilde F^{\alpha_1, \ldots, \alpha_N}_{k_1, \ldots, k_N}(q) 
&=&(q; q)_\infty^{-\chi(X)} \cdot \prod_{i=1}^N \Theta^{\alpha_i}_{k_i}(q) + W_{1, 2}
       \nonumber    \\
&=&(q; q)_\infty^{-\chi(X)} \cdot \Coe_{z_1^0 \cdots z_N^0} 
\left (\prod_{i=1}^N \Theta^{\alpha_i}_{k_i}(q, z_i) \right ) + W_{1, 2}
\end{eqnarray}
by Lemma~\ref{ThetaAlphaK}. Putting $W_1 = W_{1,1} + W_{1,2}$ completes the proof of 
\eqref{ThmFk1NAlpha1N.00}.
\end{proof}

Our next goal is to relate the lower weight term $W_{1,2}$ in \eqref{ThmFk1NAlpha1N.2}
and \eqref{ThmFk1NAlpha1N.3}
to multiple $q$-zeta values (with additional variables $z_1, \ldots, z_N$ inserted).
We will assume $e_X \alpha_i = 0$ for all $1 \le i \le N$.
We begin with a lemma strengthening Lemma~\ref{Trqdjpx}.

\begin{lemma}  \label{EX=0Trqdjpx}
Let $\la^{(1)}, \ldots, \la^{(N)} \in \Wcp$, 
and $\alpha_1, \ldots, \alpha_N \in H^*(X)$ be homogeneous. Assume that 
$e_X \alpha_i = 0$ for every $1 \le i \le N$, and $\sum_{i=1}^N |\la^{(i)}| = 0$. Put
$$
A_N = \Tr \, q^\fd \,\prod_{i=1}^N {\fa_{\la^{(i)}}(\alpha_i) \over \la^{(i)}!}.
$$
\begin{enumerate}
\item[{\rm (i)}] If $\ell(\la^{(i)}) \ge 2$ for every $1 \le i \le N$, then $A_N = 0$.

\item[{\rm (ii)}] If $A_N \ne 0$, then $A_N$ is a linear combination of the expressions:
\begin{eqnarray}        
& &(q; q)_\infty^{-\chi(X)} \cdot {\rm Sign}(\pi) \cdot
  \prod_{i=1}^u \left \langle 1_X, \prod_{j \in \pi_i} \alpha_j \right \rangle \cdot
  \prod_{i=1}^{\w \ell} {(-\w n_i)q^{\w n_i \w p_i} \over 1 - q^{\w n_i}}  
  \label{EX=0Trqdjpx.01}           \\
&=&(q; q)_\infty^{-\chi(X)} \cdot {\rm Sign}(\pi) \cdot
  \prod_{i=1}^u \left \langle 1_X, \prod_{j \in \pi_i} \alpha_j \right \rangle \cdot
  \prod_{i=1}^{\ell} {(-n_i')^{w_i} q^{n_i' p_i} \over (1 - q^{n_i'})^{w_i}}  
  \label{EX=0Trqdjpx.02}  
\end{eqnarray}
where $\w \ell = \sum_{i=1}^N \ell(\la^{(i)})/2 = \sum_{i=1}^\ell w_i$, 
$\w p_i \in \{0, 1\}$, $0 \le p_i \le w_i$,
the partition $\pi = \{\pi_1, \ldots, \pi_u\}$ of $\{1, \ldots, N\}$ depend only on 
$\la^{(1)}, \ldots, \la^{(N)}$, the integers
$\w n_1, \ldots, \w n_{\w \ell}$ are the positive parts (repeated with multiplicities) 
in $\la^{(1)}, \ldots, \la^{(N)}$, the integers $n_1', \ldots, n_{\ell}'$ denote 
the different integers in $\w n_1, \ldots, \w n_{\w \ell}$, 
and each $n_i'$ appears $w_i$ times in $\w n_1, \ldots, \w n_{\w \ell}$.
\end{enumerate}
\end{lemma}
\begin{proof}
(i) As in the proof of Lemma~\ref{Trqdjpx},
$A_N = 0$ unless $\ell(\la^{(i)}) = 2$ and $|\alpha_i| = 0$ for every $1 \le i \le N$.
Assume $\ell(\la^{(i)}) = 2$ and $|\alpha_i| = 0$ for every $1 \le i \le N$.
To prove $A_N = 0$, we will use induction on $N$.
If $N = 1$, then $A_1 = 0$ by \eqref{Trqdjpx.2}. Let $N \ge 2$. 
If $|\la^{(i)}| = 0$  for every $1 \le i \le N$, then $A_N = 0$ by \eqref{Trqdjpx.1}. 
Assume $|\la^{(i_0)}| \ne 0$ for some $1 \le i_0 \le N$.
Since $\sum_{i=1}^N |\la^{(i)}| = 0$, we may further assume that $|\la^{(i_0)}| < 0$.
By \eqref{Trqdjpx.3}, Lemma~\ref{tau_k_tau_{k-1}}~(i) and (ii), and induction,
we conclude that $A_N = 0$.

(ii) Note that \eqref{EX=0Trqdjpx.02} follows from \eqref{EX=0Trqdjpx.01} since 
each integer $n_i'$ appears $w_i$ times among the integers $\w n_1, \ldots, \w n_{\w \ell}$.
In the following, we will prove \eqref{EX=0Trqdjpx.01}. 
To simplify the signs, we will assume that $|\alpha_i|$ is even for every $i$.

Since $A_N \ne 0$, we conclude from (i) that $\ell(\la^{(i_0)}) = 1$ for 
some $1 \le i_0 \le N$. If $\la^{(i_0)} = (-n_0)$ for some $n_0 > 0$,
then by \eqref{Trqdjpx.3}, $A_N$ is equal to
$$
{1 \over 1 - q^{n_0}} \, \sum_{r = 1}^{i_0 -1} \Tr \, q^\fd 
\prod_{i=1}^{r-1} {\fa_{\la^{(i)}}(\alpha_i) \over \la^{(i)}!}
    \cdot \left [{\fa_{\la^{(r)}}(\alpha_r) \over \la^{(r)}!},
      \fa_{-n_0}(\alpha_{i_0}) \right ] 
    \cdot \prod_{r+1 \le i \le N, i \ne i_0} {\fa_{\la^{(i)}}(\alpha_i) \over \la^{(i)}!}
$$
$$
+ {q^{n_0} \over 1 - q^{n_0}} \, \sum_{r = i_0+1}^N \Tr \, q^\fd 
\prod_{1 \le i \le r-1, i \ne i_0} {\fa_{\la^{(i)}}(\alpha_i) \over \la^{(i)}!} 
     \cdot \left [{\fa_{\la^{(r)}}(\alpha_r) \over \la^{(r)}!},
     \fa_{-n_0}(\alpha_{i_0}) \right ] \cdot
     \prod_{i = r+1}^N {\fa_{\la^{(i)}}(\alpha_i) \over \la^{(i)}!}.
$$
Similarly, if $\la^{(i_0)} = (n_0)$ for some $n_0 > 0$, then $A_N$ is equal to
$$
{q^{n_0} \over 1 - q^{n_0}} \, \sum_{r = 1}^{i_0 -1} \Tr \, q^\fd 
\prod_{i=1}^{r-1} {\fa_{\la^{(i)}}(\alpha_i) \over \la^{(i)}!}
    \cdot \left [\fa_{n_0}(\alpha_{i_0}), {\fa_{\la^{(r)}}(\alpha_r) \over \la^{(r)}!} \right ] 
    \cdot \prod_{r+1 \le i \le N, i \ne i_0} {\fa_{\la^{(i)}}(\alpha_i) \over \la^{(i)}!}
$$
$$
+ {1 \over 1 - q^{n_0}} \, \sum_{r = i_0+1}^N \Tr \, q^\fd 
\prod_{1 \le i \le r-1, i \ne i_0} {\fa_{\la^{(i)}}(\alpha_i) \over \la^{(i)}!} 
     \cdot \left [\fa_{n_0}(\alpha_{i_0}), {\fa_{\la^{(r)}}(\alpha_r) \over \la^{(r)}!} \right ] 
     \cdot \prod_{i = r+1}^N {\fa_{\la^{(i)}}(\alpha_i) \over \la^{(i)}!}.
$$
Note that $[{\fa_{\la^{(r)}}(\alpha_r)/\la^{(r)}!}, \fa_{-n_0}(\alpha_{i_0})]
= (-n_0) \fa_{\la^{(r)} - (n_0)}(\alpha_r \alpha_{i_0})/\big (\la^{(r)} - (n_0) \big )!$, and
$[\fa_{n_0}(\alpha_{i_0}), {\fa_{\la^{(r)}}(\alpha_r)/\la^{(r)}!}]
= (-n_0) \fa_{\la^{(r)} - (-n_0)}(\alpha_{i_0} \alpha_r)/\big (\la^{(r)} - (-n_0) \big )!$.
Therefore, by induction, $A_N$ is a linear combination of the expressions
\eqref{EX=0Trqdjpx.01}. We remark that the negative parts (repeated with multiplicities) 
in $\la^{(1)}, \ldots, \la^{(N)}$ are $-\w n_1, \ldots, -\w n_{\w \ell}$.
\end{proof}

\begin{theorem}   \label{EX=0Fk1NAlpha1N}
For $1 \le i \le N$, let $k_i \ge 0$ and $\alpha_i \in H^*(X)$ be homogeneous. 
Assume that $e_X \alpha_i = 0$ for every $1 \le i \le N$. Then, 
\begin{eqnarray}   \label{EX=0Fk1NAlpha1N.0}
\widetilde F^{\alpha_1, \ldots, \alpha_N}_{k_1, \ldots, k_N}(q)
= (q; q)_\infty^{-\chi(X)} \cdot \Coe_{z_1^0 \cdots z_N^0} 
  \left (\prod_{i=1}^N \Theta^{\alpha_i}_{k_i}(q, z_i) \right ) + W_{1,2},
\end{eqnarray}
and $(q; q)_\infty^{\chi(X)} \cdot W_{1,2}$ is a linear combination of the coefficients 
of $z_1^0 \cdots z_N^0$ in some multiple $q$-zeta values 
(with variables $z_1, \ldots, z_N$ inserted) of weights
$< \sum_{i=1}^N (k_i +2)$. Moreover, the coefficients in this linear combination 
are independent of $q$.
\end{theorem}
\begin{proof}
To simplify the signs, we will assume that $|\alpha_i|$ is even for every $i$.
Recall that $\widetilde F^{\alpha_1, \ldots, \alpha_N}_{k_1, \ldots, k_N}(q)$
is defined in \eqref{DefWideF}, and that \eqref{EX=0Fk1NAlpha1N.0} is 
just \eqref{ThmFk1NAlpha1N.3}. From the proofs of \eqref{ThmFk1NAlpha1N.3}
and Theorem~\ref{ThmJJkAlpha}, we see that the lower weight term $W_{1,2}$
in \eqref{EX=0Fk1NAlpha1N.0} is the contributions of Case B 
in the proof of Theorem~\ref{ThmJJkAlpha} to the right-hand-side of \eqref{DefWideF}.
By \eqref{ThmJJkAlpha.3} and Lemma~\ref{ThetaAlphaK},
up to a re-ordering of the set $\{1, \ldots, N\}$,
these contributions are of the following form, denoted by $C_{2, N-N_1}$:
$$   
\Coe_{z_1^0 \cdots z_{N_1}^0} \left (\prod_{i=1}^{N_1} \Theta^{\alpha_i}_{k_i}(q, z_i) \right ) 
$$
$$
\cdot (-1)^{N-N_1} \cdot \sum_{\ell(\la^{(i)}) = k_i +2, |\la^{(i)}|=0 
     \atop N_1+1 \le i \le N} \sum_{\w \la^{(i)} < \la^{(i)} \atop N_1+1 \le i \le N}
\prod_{N_1+1 \le i \le N \atop n \ge 1} \left ( {(-1)^{p^{(i)}_{n}} \over p^{(i)}_{n}!} 
    {q^{n p^{(i)}_{n}} \over (1-q^n)^{p^{(i)}_{n}}}  {1 \over \w p^{(i)}_{n}!} 
    {1 \over (1-q^n)^{\w p^{(i)}_{n}}} \right )
$$
\begin{eqnarray*}    
\cdot     \Tr \, q^\fd \prod_{i=N_1+1}^N 
     {\fa_{\la^{(i)} - \w \la^{(i)}}
\big ((1_X - K_X)^{\sum_{n \ge 1} p^{(i)}_{n}}\alpha_i \big ) \over
     \big (\la^{(i)}-\w \la^{(i)} \big )!}
\end{eqnarray*}
where $0 \le N_1 < N$, $\w \la^{(i)}$ is denoted by $\big ( \cdots (-n)^{\w p^{(i)}_n} \cdots 
(-1)^{\w p^{(i)}_1} 1^{p^{(i)}_1} \cdots n^{p^{(i)}_n} \cdots \big )$,
and $\sum_{i = N_1+1}^N |\la^{(i)} - \w \la^{(i)}| = 0$.
We may let $N_1 = 0$. Put $\mu^{(i)} = \la^{(i)}-\w \la^{(i)}$. Then $C_{2, N}$ is
\begin{eqnarray}       \label{EX=0Fk1NAlpha1N.1}
(-1)^{N} \cdot \sum_{\ell(\w \la^{(i)}) + \ell(\mu^{(i)}) = k_i +2 
  \atop {|\w \la^{(i)}| + |\mu^{(i)}|=0 \atop 1 \le i \le N}} 
\prod_{1 \le i \le N \atop n \ge 1} \left ( {(-1)^{p^{(i)}_{n}} \over p^{(i)}_{n}!} 
    {q^{n p^{(i)}_{n}} \over (1-q^n)^{p^{(i)}_{n}}}  {1 \over \w p^{(i)}_{n}!} 
    {1 \over (1-q^n)^{\w p^{(i)}_{n}}} \right )  
\end{eqnarray}
\begin{eqnarray}    \label{EX=0Fk1NAlpha1N.100}
\cdot \Tr \, q^\fd \prod_{i=1}^N  {\fa_{\mu^{(i)}}
\big ((1_X - K_X)^{\sum_{n \ge 1} p^{(i)}_{n}}\alpha_i \big ) \over \mu^{(i)}!}
\end{eqnarray}
where $\mu^{(i)} \ne \emptyset$ for every $1 \le i \le N$, and $\sum_{i = 1}^N |\mu^{(i)}| = 0$. 
By Lemma~\ref{EX=0Trqdjpx}~(ii), the trace on line \eqref{EX=0Fk1NAlpha1N.100}
is a linear combination of the expressions:
\begin{eqnarray}       \label{EX=0Fk1NAlpha1N.2}      
(q; q)_\infty^{-\chi(X)} \cdot 
  \prod_{i=1}^u \left \langle 1_X, \prod_{j \in \pi_i} 
  \big ( (1_X - K_X)^{\sum_{n \ge 1} p^{(j)}_{n}}\alpha_j \big ) \right \rangle  \cdot
  \prod_{i=1}^{\ell} {(-n_i')^{w_i} q^{n_i' p_i} \over (1 - q^{n_i'})^{w_i}}  
\end{eqnarray}
where $\sum_{i=1}^\ell w_i = \sum_{i=1}^N \ell(\mu^{(i)})/2$, $0 \le p_i \le w_i$,
and the mutually distinct integers $n_1', \ldots, n_\ell'$ appear 
$w_1, \ldots, w_\ell$ times respectively as the positive parts 
(repeated with multiplicities) of $\mu^{(1)}, \ldots, \mu^{(N)}$
(so the negative parts, repeated with multiplicities, 
of $\mu^{(1)}, \ldots, \mu^{(N)}$ are $-n_1', \ldots, -n_\ell'$ with multiplicities 
$w_1, \ldots, w_\ell$ respectively).

We now fix the type of the $N$-tuple
$(\mu^{(1)}, \ldots, \mu^{(N)})$. Define $\mathfrak T$ to be the set consisting of 
all the $N$-tuples $(\w \mu^{(1)}, \ldots, \w \mu^{(N)})$ obtained from 
$(\mu^{(1)}, \ldots, \mu^{(N)})$ as follows: take $N$ mutually distinct positive integers
$n_1, \ldots, n_{\ell}$, and obtain $\w \mu^{(i)}, 1 \le i \le N$ from 
$\mu^{(i)}, 1 \le i \le N$ by replacing every part $\pm n_j'$ in $\mu^{(i)}$ by $\pm n_j$.
Denote the contribution of the type $\mathfrak T$ to $C_{2, N}$ by $C_{2, N}^{\mathfrak T}$.
Then, $C_{2, N} = \sum_{\mathfrak T} C_{2, N}^{\mathfrak T}$. Thus, to prove the statement 
about $(q; q)_\infty^{\chi(X)} \cdot W_{1,2}$ in our theorem, it remains to 
study $C_{2, N}^{\mathfrak T}$. For $1 \le i \le N$, let $\ell_{i, +}$ 
(resp. $\ell_{i, -}$) be the sum of the multiplicities of the positive 
(resp. negative) parts in $\mu^{(i)}$. Denote the parts (repeated with multiplicities) 
of $\mu^{(i)}$ by $-n_{j_{i,1}}', \ldots, -n_{j_{i, \ell_{i, -}}}', 
n_{h_{i,1}}', \ldots, n_{h_{i, \ell_{i, +}}}'$. By the definition of $\mathfrak T$,
the following data are the same for every $N$-tuple
$(\w \mu^{(1)}, \ldots, \w \mu^{(N)}) \in \mathfrak T$:
\begin{enumerate}
\item[$\bullet$] the indexes $j_{i,1}, \ldots, j_{i, \ell_{i, -}} (1 \le i \le N)$ 
up to re-ordering

\item[$\bullet$] the indexes $h_{i,1}, \ldots, h_{i, \ell_{i, +}} (1 \le i \le N)$ 
up to re-ordering

\item[$\bullet$] the partition $\{\pi_1, \ldots, \pi_u\}$ of $\{1, \ldots, N\}$ and integers 
$w_i, p_i$ in \eqref{EX=0Fk1NAlpha1N.2}

\item[$\bullet$] the coefficient of line \eqref{EX=0Fk1NAlpha1N.2} in the linear combination.
\end{enumerate}
So by \eqref{EX=0Fk1NAlpha1N.1}, \eqref{EX=0Fk1NAlpha1N.100} and \eqref{EX=0Fk1NAlpha1N.2}, 
$C_{2, N}^{\mathfrak T}$ is a linear combination of the expressions
\begin{eqnarray*} 
(q; q)_\infty^{-\chi(X)} 
  \sum_{n_1, \ldots, n_\ell > 0 \atop n_i \ne n_j \text{\rm if } i \ne j}   
  \prod_{i=1}^{\ell} {(-n_i)^{w_i} q^{n_i p_i} \over (1 - q^{n_i})^{w_i}}   
\sum_{\ell(\w \la^{(i)}) = k_i +2 - \ell_{i, +} - \ell_{i, -}
  \atop {|\w \la^{(i)}| = \sum_{r=1}^{\ell_{i, -}} n_{j_{i, r}} - 
  \sum_{r=1}^{\ell_{i, +}} n_{h_{i,r}} \atop 1 \le i \le N}}  
\end{eqnarray*}
$$
\prod_{i=1}^u \left \langle 1_X, \prod_{j \in \pi_i} 
  \big ( (1_X - K_X)^{\sum_{n \ge 1} p^{(j)}_{n}}\alpha_j \big ) \right \rangle 
\cdot \prod_{1 \le i \le N \atop n \ge 1} \left ( {(-1)^{p^{(i)}_{n}} \over p^{(i)}_{n}!} 
    {q^{n p^{(i)}_{n}} \over (1-q^n)^{p^{(i)}_{n}}}  {1 \over \w p^{(i)}_{n}!} 
    {1 \over (1-q^n)^{\w p^{(i)}_{n}}} \right ) 
$$
(we have moved the factor $(-1)^N$ into the coefficients of 
the linear combination). Inserting the variables $z_1, \ldots, z_N$, 
we conclude that $(q; q)_\infty^{\chi(X)} \cdot C_{2, N}^{\mathfrak T}$ is 
a linear combination of the coefficients of $z_1^0 \cdots z_N^0$ in the expressions
\begin{eqnarray}    \label{EX=0Fk1NAlpha1N.200}
\left ( \sum_{n_1, \ldots, n_\ell > 0 \atop n_i \ne n_j \text{\rm if } i \ne j}   
  \prod_{i=1}^{\ell} {(-n_i)^{w_i} q^{n_i p_i} \over (1 - q^{n_i})^{w_i}} 
  \cdot \prod_{1 \le i \le N} z_i^{-\sum_{r=1}^{\ell_{i, -}} n_{j_{i, r}} + 
  \sum_{r=1}^{\ell_{i, +}} n_{h_{i,r}}} \right ) 
\end{eqnarray}
\begin{eqnarray}       \label{EX=0Fk1NAlpha1N.3}
\cdot \sum_{\ell(\w \la^{(i)}) = k_i +2 - \ell_{i, +} - \ell_{i, -} \atop 1 \le i \le N} 
\prod_{i=1}^u \left \langle 1_X, \prod_{j \in \pi_i} 
  \big ( (1_X - K_X)^{\sum_{n \ge 1} p^{(j)}_{n}}\alpha_j \big ) \right \rangle 
\end{eqnarray}
\begin{eqnarray}       \label{EX=0Fk1NAlpha1N.4}
\cdot \prod_{1 \le i \le N \atop n \ge 1} \left ( {(-1)^{p^{(i)}_{n}} \over p^{(i)}_{n}!} 
    {(qz_i)^{n p^{(i)}_{n}} \over (1-q^n)^{p^{(i)}_{n}}}  {1 \over \w p^{(i)}_{n}!} 
    {z_i^{-n \w p^{(i)}_{n}} \over (1-q^n)^{\w p^{(i)}_{n}}} \right ). 
\end{eqnarray}
We claim that line \eqref{EX=0Fk1NAlpha1N.200} is the sum of $\ell!$ multiple $q$-zeta values
of weight $\sum_{i=1}^\ell w_i$. Indeed, the sum of the terms with 
$n_1 > \cdots > n_\ell$ in line \eqref{EX=0Fk1NAlpha1N.200} is equal to:
\begin{eqnarray}    \label{EllFactorial1}
& &\sum_{n_1 > \cdots > n_\ell}   
  \prod_{i=1}^{\ell} {(-n_i)^{w_i} q^{n_i p_i} \over (1 - q^{n_i})^{w_i}} 
  \cdot \prod_{1 \le i \le N} z_i^{-\sum_{r=1}^{\ell_{i, -}} n_{j_{i, r}} + 
  \sum_{r=1}^{\ell_{i, +}} n_{h_{i,r}}}     \nonumber    \\
&=&\sum_{n_1 > \cdots > n_\ell} \prod_{i=1}^{\ell} 
  {(-n_i)^{w_i} q^{n_i p_i} f_i(z_1, \ldots, z_N)^{n_i} \over (1 - q^{n_i})^{w_i}} 
\end{eqnarray}
where each $f_i(z_1, \ldots, z_N)$ is a suitable monomial of $z_1^{\pm 1}, \ldots, z_N^{\pm 1}$.
So line \eqref{EX=0Fk1NAlpha1N.200} is the sum of the following $\ell!$ multiple $q$-zeta values:
\begin{eqnarray}    \label{EllFactorial2}
\sum_{n_1 > \cdots > n_\ell} \prod_{i=1}^{\ell} {(-n_i)^{w_{\sigma(i)}} q^{n_i p_{\sigma(i)}} 
f_{\sigma(i)}(z_1, \ldots, z_N)^{n_i} \over (1 - q^{n_i})^{w_{\sigma(i)}}} 
\end{eqnarray}
where $\sigma$ runs in the symmetric group $S_\ell$. Furthermore,
as in the proof of Lemma~\ref{ThetaAlphaK}, the product of lines \eqref{EX=0Fk1NAlpha1N.3} 
and \eqref{EX=0Fk1NAlpha1N.4} is equal to
\begin{eqnarray*}        
\sum_{a_i, b_i \ge 0; s_1^{(i)}, \ldots, s_{a_i}^{(i)}, t_1^{(i)}, \ldots, t_{b_i}^{(i)} \ge 1
  \atop {\sum_{r=1}^{a_i} s_r^{(i)} + \sum_{r=1}^{b_i} t_r^{(i)} 
    = k_i +2 - \ell_{i, +} - \ell_{i, -} \atop 1 \le i \le N}}
\prod_{i=1}^u \left \langle 1_X, \prod_{j \in \pi_i} 
  \big ( (1_X - K_X)^{\sum_{r=1}^{a_j} s_r^{(j)}}\alpha_j \big ) \right \rangle  
\cdot \prod_{1 \le r \le a_i \atop 1 \le i \le N} {(-1)^{s_r^{(i)}} \over s_r^{(i)}!}   
\end{eqnarray*}
\begin{eqnarray*} 
\cdot \prod_{1 \le r \le b_i \atop 1 \le i \le N} {1 \over t_r^{(i)}!} 
\cdot \prod_{i=1}^N \left ( \sum_{n_1 > \cdots > n_{a_i}} 
   \prod_{r= 1}^{a_i} {(q z_i)^{n_r s_r^{(i)}} \over (1-q^{n_r})^{s_r^{(i)}}}
\cdot \sum_{m_1 > \cdots > m_{b_i}} \prod_{r= 1}^{b_i} {z_i^{-m_r t_r^{(i)}} 
   \over (1-q^{m_r})^{t_r^{(i)}}} \right ). 
\end{eqnarray*}
Combining with lines \eqref{EX=0Fk1NAlpha1N.200} and \eqref{EllFactorial2}, 
we see that $(q; q)_\infty^{\chi(X)} \cdot C_{2, N}^{\mathfrak T}$ 
is a linear combination of the coefficients of $z_1^0 \cdots z_N^0$ in 
some multiple $q$-zeta values of weights
\begin{eqnarray*}
    w 
&:=&\sum_{i=1}^{\ell} w_i + \sum_{i=1}^N \left ( \sum_{r= 1}^{a_i} s_r^{(i)}
       + \sum_{r= 1}^{b_i} t_r^{(i)} \right ) \\
&=& \sum_{i=1}^{\ell} w_i + \sum_{i=1}^N (k_i +2 - \ell_{i, +} - \ell_{i, -}).
\end{eqnarray*}
Note from \eqref{EX=0Fk1NAlpha1N.2} that $\sum_{i=1}^{\ell} w_i 
= \sum_{i=1}^N \ell(\mu^{(i)})/2 = \sum_{i=1}^N (\ell_{i, +} + \ell_{i, -})/2$.
So we have $w < \sum_{i=1}^N (k_i +2)$. This completes the proof of our theorem.
\end{proof}

We will end this section with three propositions about $F^\alpha_{k}(q)$,
which provide some insight into the lower weight term $W_1$ in Theorem~\ref{ThmFk1NAlpha1N}.
Proposition~\ref{PropFq0Alpha} deals with $F^{\alpha}_{0}(q)$ for an arbitrary $\alpha \in H^*(X)$.
Proposition~\ref{Propch1Alpha} calculates $F^{\alpha}_{1}(q)$ by assuming $e_X \alpha = 0$.
Proposition~\ref{ThmFkAlpha} computes $F^{\alpha}_{k}(q), k \ge 2$ by assuming 
$e_X \alpha = K_X \alpha = 0$. 

\begin{proposition}  \label{PropFq0Alpha}
The generating series $F^{\alpha}_{0}(q)$ is equal to
\begin{eqnarray}   \label{PropFq0Alpha.0}
(q; q)_\infty^{-\chi(X)} \cdot \langle 1_X-K_X, \alpha \rangle
  \cdot \sum_{n} { q^n \over (1-q^n)^2} 
+ (q; q)_\infty^{-\chi(X)} \cdot \langle e_X, \alpha\rangle 
  \cdot \sum_{n} {nq^n \over 1-q^n}.   \quad
\end{eqnarray}
\end{proposition}
\begin{proof}
By Lemma~\ref{FtoW}, $F_k^\alpha(q) = \Tr \, q^\fd \, W(\fL_1, z) \, \fG_k(\alpha)$.
By Theorem~\ref{char_th}, we have 
$\mathfrak G_0(\alpha) = -\sum_{n > 0} (\fa_{-n} \fa_n)(\alpha)$.
Now \eqref{PropFq0Alpha.0} follows from Remark~\ref{Rmk-914}.
\end{proof}

\begin{remark}  \label{RmkF01X}
By \eqref{PropFq0Alpha.0}, $\displaystyle{F^{1_X}_{0}(q) 
= (q; q)_\infty^{-\chi(X)} \cdot \chi(X) \cdot \sum_{n} {nq^n \over 1-q^n}
= q {{\rm d} \over {\rm d}q} (q; q)_\infty^{-\chi(X)}}$.
\end{remark}

\begin{proposition}   \label{Propch1Alpha}
Let $\alpha \in H^*(X)$ be a homogeneous class satisfying $e_X \alpha = 0$. Then, 
the generating series $F_1^\alpha(q)$ is the coefficient of $z^0$ in
$$ 
(q; q)_\infty^{-\chi(X)} \cdot {\langle K_X - K_X^2, \alpha \rangle \over 2} \cdot 
$$
$$
\cdot \left ( \sum_{n} {(n-1)q^n \over (1 - q^n)^2} 
+ \sum_n {(qz)^{n} \over 1-q^n} 
\cdot \left (\sum_m {z^{-2m} \over (1-q^m)^2} + 2 \sum_{m_1 > m_2} 
{z^{-m_1} \over 1-q^{m_1}} {z^{-m_2} \over 1-q^{m_2}} \right ) \right ).
$$
\end{proposition}
\begin{proof}
We have $F_1^\alpha(q) = \Tr \, q^\fd \, W(\fL_1, z) \, \fG_1(\alpha)$.
It is known that
\begin{eqnarray}  \label{char_thK=1.0}
  \mathfrak G_1(\alpha) 
= -\sum_{\ell(\lambda) = 3, |\lambda|=0}
  {\mathfrak a_{\lambda}(\alpha) \over \lambda!} - \sum_{n>0} 
  \frac{n-1}2 (\mathfrak a_{-n} \mathfrak a_n)(K_X \alpha).      
\end{eqnarray}
Applying Remark~\ref{Rmk-914} to $\displaystyle{-\sum_{n>0} 
\frac{n-1}2 \Tr \, q^\fd \, W(\fL_1, z) \, 
(\mathfrak a_{-n} \mathfrak a_n)(K_X \alpha)}$ yields the weight-$2$ terms in our proposition.
Again by Remark~\ref{Rmk-914}, the trace $\displaystyle{\Tr \, q^\fd \, W(\fL_1, z) \, 
{\mathfrak a_{\lambda}(\alpha) \over \lambda!}}$ with $\ell(\lambda) = 3$ and $|\lambda|=0$
contains only weight-$3$ terms (i.e., does not contain lower weight terms).
So the proof of Theorem~\ref{ThmFk1NAlpha1N} shows that 
$$
-\sum_{\ell(\lambda) = 3, |\lambda|=0} \Tr \, q^\fd \, W(\fL_1, z) \, 
  {\mathfrak a_{\lambda}(\alpha) \over \lambda!} 
= (q; q)_\infty^{-\chi(X)} \cdot \Coe_{z^0} \Theta^{\alpha}_{1}(q, z).
$$
Expanding $\Coe_{z^0} \Theta^{\alpha}_{1}(q, z)$ yields the weight-$3$ terms in our proposition.
\end{proof}

\begin{proposition}   \label{ThmFkAlpha}
Let $\alpha \in H^*(X)$ be homogeneous satisfying $K_X \alpha = e_X \alpha = 0$.
\begin{enumerate}
\item[{\rm (i)}] If $|\alpha| < 4$, then $F^\alpha_{k}(q) = 0$ for every $k \ge 0$;

\item[{\rm (ii)}] Let $|\alpha| = 4$ and $k \ge 0$. 
Then, $F^\alpha_{k}(q)$ is the coefficient of $z^0$ in
$$
-(q; q)_\infty^{-\chi(X)} \cdot \langle 1_X, \alpha \rangle \cdot
\sum_{a, s_1, \ldots, s_a, b, t_1, \ldots, t_b \ge 1 
\atop \sum_{i=1}^a s_i + \sum_{j = 1}^b t_j = k+2} 
\prod_{i=1}^a {(-1)^{s_i} \over s_i!} \cdot \prod_{j=1}^b {1 \over t_j!} 
$$
\begin{eqnarray}    \label{ThmFkAlpha.0}
\cdot \sum_{n_1 > \cdots > n_a} \prod_{i= 1}^a {(q z)^{n_i s_i} \over (1-q^{n_i})^{s_i}}
\cdot \sum_{m_1 > \cdots > m_b} \prod_{j= 1}^b {z^{-m_jt_j} \over (1-q^{m_j})^{t_j}}.
\end{eqnarray}
In particular, if $2 \not |k$, then $F^\alpha_{k}(q) = 0$.
\end{enumerate}
\end{proposition}
\begin{proof}
Since $K_X \alpha = e_X \alpha = 0$, we conclude from Theorem~\ref{char_th} that 
\begin{eqnarray*}
\mathfrak G_k(\alpha) = 
-\sum_{\ell(\lambda) = k+2, |\lambda|=0} {\mathfrak a_{\lambda}(\alpha) \over \lambda!}. 
\end{eqnarray*}
As in the proof of Proposition~\ref{Propch1Alpha}, 
Remark~\ref{Rmk-914} and the proof of Theorem~\ref{ThmFk1NAlpha1N} yield
\begin{eqnarray*}    
F^{\alpha}_{k}(q) =
(q; q)_\infty^{-\chi(X)} \cdot \Coe_{z^0} \Theta^{\alpha}_{k}(q, z).
\end{eqnarray*}
By the definition of $\Theta^{\alpha}_{k}(q, z)$ in \eqref{ThetaAlphaK.02}, 
we see that (i) holds and that our formula for $F^\alpha_{k}(q)$ with 
$|\alpha| = 4$ and $k \ge 0$ holds.
Note that line \eqref{ThmFkAlpha.0} can be rewritten as
$$
\sum_{n_1 > \cdots > n_a} \prod_{i= 1}^a {(q z^{2})^{n_i s_i/2} \over (1-q^{n_i})^{s_i}}
\cdot \sum_{m_1 > \cdots > m_b} \prod_{j= 1}^b {(q z^{-2})^{m_j t_j/2} \over (1-q^{m_j})^{t_j}}.
$$
Therefore, if $|\alpha| = 4$ and $2 \nmid k$, then the role of $a, s_1, \ldots, s_a$ and 
the role of $b, t_1, \ldots, t_b$ in the above formula of $F^\alpha_{k}(q)$ 
are anti-symmetric; so $F^\alpha_{k}(q) = 0$.
\end{proof}

\section{\bf The reduced series $\big \langle \ch_{k_1}^{L_1} \cdots \ch_{k_N}^{L_N} \big \rangle'$} 
\label{sect_chk}

In this section, we will prove Conjecture~\ref{OkoConj} modulo the lower weight term. 
Moreover, for abelian surfaces, we will verify Conjecture~\ref{OkoConj}.

Let $L$ be a line bundle on the smooth projective surface $X$. 
It induces the tautological rank-$n$ bundle $\Ln$ over the Hilbert scheme $\Xn$:
$$
\Ln = p_{1*}\big (p_2^*L|_{\mathcal Z_n} \big )
$$
where $\mathcal Z_n$ is the universal codimension-$2$ subscheme of $\Xn\times X$,
and $p_1$ and $p_2$ are the projections of $\Xn \times X$ to $\Xn$ and $X$ respectively.
By the Grothendieck-Riemann-Roch Theorem and \eqref{DefOfGGammaN}, we obtain
\begin{eqnarray}     \label{chLnGRR}
   \ch (\Ln) 
&=&p_{1*}(\ch({\mathcal O}_{{\mathcal Z}_n}) \cdot p_2^*\ch(L) \cdot p_2^*{\rm td}(X) )
               \nonumber    \\
&=&G(1_X, n) + G(L, n) + G(L^2/2, n).
\end{eqnarray}
Since the cohomology degree of $G_i(\alpha, n)$ is $2i+ |\alpha|$, we have
\begin{eqnarray}     \label{chLnGAlpha}
\ch_k (\Ln) = G_k(1_X, n) + G_{k-1}(L, n) + G_{k-2}(L^2/2, n).
\end{eqnarray} 

Following Okounkov \cite{Oko}, we have defined the generating series 
$\big \langle \ch_{k_1}^{L_1} \cdots \ch_{k_N}^{L_N} \big \rangle$ 
and its reduced version $\big \langle \ch_{k_1}^{L_1} \cdots \ch_{k_N}^{L_N} \big \rangle'$
in \eqref{OkoChkN.1} and \eqref{OkoChkN.2} respectively.

\begin{theorem}   \label{AbelianSur}
Let $L_1, \ldots, L_N$ be line bundles over $X$, and $k_1, \ldots, k_N \ge 0$. Then, 
\begin{eqnarray}         \label{AbelianSur.0}
\big \langle \ch_{k_1}^{L_1} \cdots \ch_{k_N}^{L_N} \big \rangle'
= \Coe_{z_1^0 \cdots z_N^0} \left (\prod_{i=1}^N \Theta^{1_X}_{k_i}(q, z_i) \right ) + W,
\end{eqnarray}
and the lower weight term $W$ is an infinite linear combination of the expressions:
\begin{eqnarray*}      
\prod_{i=1}^u \left \langle K_X^{r_i}e_X^{r_i'}, L_1^{\ell_{i, 1}}  \cdots 
   L_N^{\ell_{i, N}} \right \rangle 
\cdot \prod_{i=1}^v {q^{n_i w_ip_i} \over (1-q^{n_i})^{w_i}}
\end{eqnarray*} 
where $\sum_{i=1}^v w_i < \sum_{i=1}^N (k_i + 2)$, 
and the integers $u, v$, $r_i, r_i', \ell_{i, j} \ge 0, 
n_i > 0, w_i >0, p_i \in \{0, 1\}$ depend only on $k_1, \ldots, k_N$.
Furthermore, all the coefficients of this linear combination are 
independent of $q, L_1, \ldots, L_N$ and $X$.
\end{theorem}
\begin{proof}
We conclude from \eqref{OkoChkN.1}, \eqref{OkoChkN.2}, \eqref{chLnGAlpha} 
and \eqref{F-generating} that 
\begin{eqnarray*}
\big \langle \ch_{k_1}^L \cdots \ch_{k_N}^L \big \rangle'
= (q; q)_\infty^{\chi(X)} \cdot F^{1_X, \ldots, 1_X}_{k_1, \ldots, k_N}(q) 
  + (q; q)_\infty^{\chi(X)} \cdot A 
\end{eqnarray*}
where $A$ is the sum of the series $F^{\alpha_1, \ldots, \alpha_N}_{k_1', \ldots, k_N'}$
such that for every $1 \le i \le N$,
$$
(\alpha_i, k_i') \in \{ (1_X, k_i), (L_i, k_i-1), (L_i^2/2, k_i-2)\},
$$ 
and $\sum_{i=1}^N k_i' < \sum_{i=1}^N k_i$.
Now our result follows from Theorem~\ref{ThmFk1NAlpha1N}.
\end{proof}

\begin{theorem}   \label{KXEX=0}
Let $L_1, \ldots, L_N$ be line bundles over an abelian surface $X$, 
and $k_1, \ldots, k_N \ge 0$. Then, the lower weight term $W$ in \eqref{AbelianSur.0}
is a linear combination of the coefficients of $z_1^0 \cdots z_N^0$ in 
some multiple $q$-zeta values (with additional variables $z_1, \ldots, z_N$ inserted) 
of weights $< \sum_{i=1}^N (k_i +2)$. 
Moreover, the coefficients in this linear combination are independent of $q$.
\end{theorem}
\begin{proof}
Since $e_X = K_X = 0$, 
$F^{\w \alpha_1, \ldots, \w \alpha_N}_{\w k_1, \ldots, \w k_N} 
= \widetilde F^{\w \alpha_1, \ldots, \w \alpha_N}_{\w k_1, \ldots, \w k_N}$ 
by Lemma~\ref{FtoW}, Theorem~\ref{char_th} and \eqref{DefWideF}. 
By Theorem~\ref{EX=0Fk1NAlpha1N} and the proof of Theorem~\ref{AbelianSur},
our theorem follows.
\end{proof}

Our next two propositions compute the series $\big \langle \ch_k^L \big \rangle$ completely,
and should offer some insight into the lower weight term $W$ in Theorem~\ref{AbelianSur}. 
Proposition~\ref{ch1L} calculates $\big \langle \ch_1^L \big \rangle$ 
by assuming $e_X = 0$, while Proposition~\ref{chkL} deals with the series 
$\big \langle \ch_k^L \big \rangle, k \ge 2$ by assuming $e_X = K_X = 0$
(i.e., by assuming that $X$ is an abelian surface). 
Note from \eqref{chLnGAlpha} that when $\chi(X) = 0$, we have
\begin{eqnarray}    \label{chktoFalpha}
\big \langle \ch_{k}^L \big \rangle' = \big \langle \ch_{k}^L \big \rangle
= F^{1_X}_k(q) + F^{L}_{k-1}(q) + {1 \over 2} \cdot F^{L^2}_{k-2}(q).
\end{eqnarray}

\begin{proposition}   \label{ch1L}
Let $L$ be a line bundle over a smooth projective surface $X$ with $e_X = 0$. Then,
the series $\big \langle \ch_{1}^L \big \rangle$ is the coefficient of $z^0$ in
$$
-\langle K_X, L \rangle \cdot \sum_{n} { q^n \over (1-q^n)^2}
- {\langle K_X, K_X \rangle \over 2} \cdot \sum_{n} {(n-1)q^n \over (1 - q^n)^2} 
$$
$$ 
- {\langle K_X, K_X \rangle \over 2}
\cdot \sum_n {(qz)^{n} \over 1-q^n} 
\cdot \left (\sum_m {z^{-2m} \over (1-q^m)^2} + 2 \sum_{m_1 > m_2} 
{z^{-m_1} \over 1-q^{m_1}} {z^{-m_2} \over 1-q^{m_2}} \right ).
$$
\end{proposition}
\begin{proof}
Our formula follows from \eqref{chktoFalpha}, \eqref{PropFq0Alpha.0} 
and Proposition~\ref{Propch1Alpha}.
\end{proof}

\begin{proposition}   \label{chkL}
Let $L$ be a line bundle over an abelian surface $X$. If $2 \nmid k$, 
then $\big \langle \ch_{k}^L \big \rangle' = 0$. If $2|k$, the generating series 
$\big \langle \ch_{k}^L \big \rangle'$ is the coefficient of $z^0$ in
$$
-{\langle L, L \rangle \over 2} \cdot
\sum_{a, s_1, \ldots, s_a, b, t_1, \ldots, t_b \ge 1 
\atop \sum_{i=1}^a s_i + \sum_{j = 1}^b t_j = k} 
\prod_{i=1}^a {(-1)^{s_i} \over s_i!} \cdot \prod_{j=1}^b {1 \over t_j!} 
$$
\begin{eqnarray*}    
\cdot \sum_{n_1 > \cdots > n_a} \prod_{i= 1}^a {(q z)^{n_i s_i} \over (1-q^{n_i})^{s_i}}
\cdot \sum_{m_1 > \cdots > m_b} \prod_{j= 1}^b {z^{-m_j t_j} \over (1-q^{m_j})^{t_j}}.
\end{eqnarray*}
\end{proposition}
\begin{proof}
Follows immediately from \eqref{chktoFalpha} and Proposition~\ref{ThmFkAlpha}.
\end{proof}
\section{\bf Applications to the universal constants in 
$\displaystyle{\sum_n c\big (T_\Xn \big )} \, q^n$} 
\label{sect_app}

Let $x \in H^4(X)$ be the cohomology class of a point in the surface $X$.
In this section, we will compute $F^{x, \ldots, x}_{k_1, \ldots, k_N}(q)$ in terms of 
the universal constants in the expression of 
$\displaystyle{\sum_n c\big (T_\Xn \big )} \, q^n$ formulated 
in \cite{Boi, BN}. Comparing with Proposition~\ref{ThmFkAlpha}~(ii) enables us 
to determine some of these universal constants.

Let $C_i = {2i \choose i}/(i+1)$ be the Catalan number and
$\sigma_1(i) = \sum_{j|i} j$.
Recall that $\cp = \Wcp_+$ is the set of all the usual partitions.
The following lemma is from \cite{Boi, BN}.

\begin{lemma} \label{Boi}
There exist unique rational numbers $b_\mu, f_\mu, g_\mu, h_\mu$ 
depending only on the partitions $\mu \in \cp$ such that 
$\displaystyle{\sum_n c\big (T_\Xn \big )} \, q^n$ is equal to 
$$
\exp \left ( \sum_{\mu \in \cp} q^{|\mu|} \Big (b_\mu \fa_{-\mu}(1_X) 
+ f_\mu \fa_{-\mu}(e_X) + g_\mu \fa_{-\mu}(K_X) 
+ h_\mu \fa_{-\mu}(K_X^2) \Big ) \right ) \vac.
$$
In addition, for $i \ge 1$, we have $b_{2i} = 0$, $b_{2i-1} = (-1)^{i-1}C_{i-1}/(2i-1)$, 
$b_{(1^i)} = f_{(1^i)} = -g_{(1^i)} = \sigma_1(i)/i$, and $h_{(1^i)} = 0$.
\end{lemma}

Our goal is to compute $F^{\alpha_1, \ldots, \alpha_N}_{k_1, \ldots, k_N}(q)$
in terms of the universal constants $b_\mu, f_\mu, g_\mu$ and $h_\mu$.
Using the definition of the operators $\fG_{k_i}(\alpha_i)$, we see that 
\begin{eqnarray}    \label{F-generating.1}
   F^{\alpha_1, \ldots, \alpha_N}_{k_1, \ldots, k_N}(q)
&=&\sum_n q^n \left \langle \left ( \prod_{i=1}^N G_{k_i}(\alpha_i, n) \right ) 
     c\big (T_\Xn \big ), 1_\Xn \right \rangle     \nonumber \\
&=&\sum_n q^n \left \langle \left ( \prod_{i=1}^N \fG_{k_i}(\alpha_i) \right ) 
     c\big (T_\Xn \big ), 1_\Xn \right \rangle     \nonumber \\
&=&\left \langle \left ( \prod_{i=1}^N \fG_{k_i}(\alpha_i) \right ) 
     \sum_n c\big (T_\Xn \big )q^n, |1\rangle \right \rangle
\end{eqnarray}
where we have put $\displaystyle{|1\rangle = \sum_n 1_\Xn 
= \exp{\big( \fa_{-1}(1_X) \big )} \cdot \vac}$.

\begin{lemma}  \label{splitting}
Let $w \in \fock$, and $\fG$ be a (possibly infinite) sum of monomials of Heisenberg creation operators.
Then, $\big \langle \fG w,|1\rangle \big \rangle =
\big \langle \fG \vac,|1\rangle \big \rangle \cdot \big \langle w,|1\rangle \big \rangle$.
\end{lemma}
\begin{proof}
By linearity, it suffices to prove that
\begin{eqnarray*}   
\left \langle \prod_{i=1}^s \fa_{-n_i}(\alpha_i) \cdot 
  \prod_{j=1}^t \fa_{-m_j}(\beta_j)\vac,|1\rangle \right \rangle 
= \left \langle \prod_{i=1}^s \fa_{-n_i}(\alpha_i) \vac,|1\rangle \right \rangle \cdot
  \left \langle \prod_{j=1}^t \fa_{-m_j}(\beta_j)\vac,|1\rangle \right \rangle
\end{eqnarray*}
where $n_1, \ldots, n_s, m_1, \ldots, m_t > 0$, 
and $\alpha_1, \ldots, \alpha_s, \beta_1, \ldots, \beta_t$ are homogeneous.
Indeed, if $\fa_{-n_i}(\alpha_i) \not \in \C \, \fa_{-1}(x)$ for some $i$ or 
if $\fa_{-m_j}(\beta_j) \not \in \C \, \fa_{-1}(x)$ for some $j$,
then both sides are equal to $0$. Otherwise, letting $\fa_{-n_i}(\alpha_i) 
= u_i \fa_{-1}(x)$ for every $i$ and $\fa_{-m_j}(\beta_j) = v_j \fa_{-1}(x)$ for every $j$,
we see that both sides are $u_1 \cdots u_s v_1 \cdots v_t$.
\end{proof}

\begin{lemma}  \label{Fqxk}
Let $b_{(1^{j})}$ and $b_{(i, 1^j)}$ be from Lemma~\ref{Boi}.
Let $\w b_{(i, 1^j)} = (j+1) b_{(1^{j+1})}$ if $i = 1$, and $\w b_{(i, 1^j)} 
= b_{(i, 1^j)}$ if $i > 1$. Then, $F^{x,\ldots,x}_{k_1, \ldots, k_N}(q)$ is equal to
$$   
(q; q)_\infty^{-\chi(X)} \cdot 
(-1)^N  \sum_{\sum_{s \ge 1} (s+1) m_{i,s} = k_i + 2 \atop 1 \le i \le N}    
\prod_{i=1}^N {1 \over (\sum_{s \ge 1} s m_{i,s})!} \cdot  
$$
$$
\cdot \prod_{s \ge 1} \left ({(-s)^{m_s} m_s! \over \prod_{i=1}^{N} m_{i,s}!} 
      \sum_{t_0 + t_1 + \cdots + t_j + \cdots = m_s} 
     \prod_{j=0}^{+\infty} {(\w b_{(s, 1^j)} q^{s+j})^{t_j} \over  t_j!} \right )
$$
where $m_{i,s} \ge 0$ for every $i$ and $s$, and $m_s = \sum_{i=1}^{N} m_{i,s}$ 
for every $s \ge 1$.
\end{lemma}
\begin{proof}
By \eqref{F-generating.1} and Theorem~\ref{char_th}, we obtain 
\begin{eqnarray}   
   F^{x,\ldots,x}_{k_1, \ldots, k_N}(q)
&=&\left \langle \left ( \prod_{i=1}^N \fG_{k_i}(x) \right ) 
     \sum_n c\big (T_\Xn \big )q^n, |1\rangle \right \rangle,  \label{Fqxk.1}  \\
\prod_{i=1}^N \mathfrak G_{k_i}(x) 
&=&(-1)^N \sum_{\ell(\lambda^{(i)}) = k_i+2, |\lambda^{(i)}|=0 \atop 1 \le i \le N}
   \prod_{i=1}^N {\mathfrak a_{\lambda^{(i)}}(x) \over (\lambda^{(i)})!}.   \label{Fqxk.2}
\end{eqnarray}
Note that $\tau_{\ell *} x = \underbrace{x \otimes \cdots \otimes x}_{\text{$\ell$ times}}$,
$K_X^2 = \langle K_X, K_X \rangle x$, $e_X = \chi(X) x$, and 
\begin{eqnarray}    \label{tauk1X}
\tau_{\ell*}1_X =
1_X \otimes \underbrace{x \otimes \cdots \otimes x}_{\text{$(\ell-1)$ times}} 
+ \cdots + \underbrace{x \otimes \cdots \otimes x}_{\text{$(\ell-1)$ times}} \otimes 1_X + w
\end{eqnarray}
where $w$ is a sum of cohomology classes of the form 
$\alpha_1 \otimes \cdots \otimes \alpha_\ell$ with $0 < |\alpha_i| < 4$ for some $i$.
So for a generalized partition $\la$, positive integers
$n_1, \ldots, n_s$ and homogeneous classes $\alpha_1, \ldots, \alpha_s \in H^*(X)$, we have
\begin{eqnarray}  \label{taukx0}   
\big \langle 
\mathfrak a_{\lambda}(x)\fa_{-n_1}(\alpha_1)\cdots \fa_{-n_s}(\alpha_s)\vac, 
|1\rangle \big \rangle = 0
\end{eqnarray}  
if $0 < |\alpha_i| < 4$ for some $i$, or if $\fa_{-n_i}(\alpha_i) \in \C \, \fa_{-j}(x)$
for some $i$ and for some $j > 1$. Combining with \eqref{Fqxk.1}, \eqref{Fqxk.2} and 
Lemma~\ref{Boi}, we see that $F^{x,\ldots,x}_{k_1, \ldots, k_N}(q)$ equals
\begin{eqnarray*}   
& &\left \langle \left ( \prod_{i=1}^N \fG_{k_i}(x) \right ) 
   \exp \left ( \sum_{\mu \in \cp}  b_\mu \fa_{-\mu}(1_X) q^{|\mu|}
   + \sum_i \w f_{(1^i)} \fa_{-(1^i)}(x)q^i  \right ) \vac, |1\rangle \right \rangle  \\
&=&\left \langle \exp \left ( \sum_i \w f_{(1^i)} \fa_{-(1^i)}(x)q^i  \right ) 
   \cdot \prod_{i=1}^N \fG_{k_i}(x) \cdot
   \exp \left ( \sum_{\mu \in \cp}  b_\mu \fa_{-\mu}(1_X) q^{|\mu|}
     \right ) \vac, |1\rangle \right \rangle
\end{eqnarray*}
where $\w f_{(1^i)} = \chi(X) \cdot f_{(1^i)}$.
By Lemma~\ref{splitting}, $F^{x,\ldots,x}_{k_1, \ldots, k_N}(q)$ 
is equal to
\begin{eqnarray}   \label{Fqxk.3}
& &\left \langle \exp \left ( \sum_i \w f_{(1^i)} \fa_{-(1^i)}(x)q^i  \right ) 
   \vac, |1\rangle \right \rangle  \nonumber   \\
&\cdot &\left \langle \prod_{i=1}^N \fG_{k_i}(x) \cdot
   \exp \left ( \sum_{\mu \in \cp}  b_\mu \fa_{-\mu}(1_X) q^{|\mu|}
     \right ) \vac, |1\rangle \right \rangle.
\end{eqnarray}
In particular, setting $N = 0$, we conclude that 
\begin{eqnarray}   \label{Fqxk.4}
\left \langle \exp \left ( \sum_i \w f_{(1^i)} \fa_{-(1^i)}(x)q^i  \right ) 
\vac, |1\rangle \right \rangle = F(q) = (q; q)_\infty^{-\chi(X)}.
\end{eqnarray}
It follows from \eqref{Fqxk.3} and \eqref{tauk1X} that 
$F^{x,\ldots,x}_{k_1, \ldots, k_N}(q)$ is equal to
\begin{eqnarray}   \label{Fqxk.5}
& &(q; q)_\infty^{-\chi(X)}  \left \langle \prod_{i=1}^N \fG_{k_i}(x) \cdot
   \exp \left ( \sum_{\mu \in \cp}  b_\mu \fa_{-\mu}(1_X) q^{|\mu|}
     \right ) \vac, |1\rangle \right \rangle   \nonumber   \\
&=&(q; q)_\infty^{-\chi(X)} \left \langle \prod_{i=1}^N \fG_{k_i}(x) \cdot
   \exp \left ( \sum_{i \ge 1 \atop j\ge 0} \w b_{(i, 1^j)} \fa_{-i}(1_X)\fa_{-1}(x)^j q^{i+j}
     \right ) \vac, |1\rangle \right \rangle   \quad \quad       
\end{eqnarray}
where $\w b_{(i, 1^j)} = (j+1) b_{(1^{j+1})}$ if $i = 1$, and $\w b_{(i, 1^j)} 
= b_{(i, 1^j)}$ if $i > 1$.
Let $\la^{(1)}, \ldots, \la^{(N)}$ be from the right-hand-side of \eqref{Fqxk.2}.
In order to have a nonzero pairing
\begin{eqnarray}   \label{Fqxk.6}
\left \langle \prod_{i=1}^N \mathfrak a_{\lambda^{(i)}}(x) \cdot
   \exp \left ( \sum_{i \ge 1, j\ge 0} \w b_{(i, 1^j)} \fa_{-i}(1_X)\fa_{-1}(x)^j q^{i+j}
     \right ) \vac, |1\rangle \right \rangle,
\end{eqnarray}
each $\la^{(i)}$ with $1 \le i \le N$ must be of the form 
$\big ( (-1)^{n_i} 1^{m_{i,1}} 2^{m_{i,2}} \cdots \big )$; 
since $\ell(\lambda^{(i)}) = k_i+2$ and $|\lambda^{(i)}|=0$, we get
$n_i + \sum_{s \ge 1} m_{i,s} = k_i + 2$ and $n_i = \sum_{s \ge 1} s m_{i, s}$; so
\begin{eqnarray}   \label{Fqxk.7}
\sum_{s \ge 1} (s+1) m_{i,s} = k_i + 2.
\end{eqnarray}
In this case, using Lemma~\ref{splitting}, we see that \eqref{Fqxk.6} is equal to
\begin{eqnarray*}   
& &\left \langle \fa_{-1}(x)^{\sum_i n_i} \cdot \prod_{i, s} \fa_s(x)^{m_{i,s}} \cdot
   \exp \left ( \sum_{i \ge 1, j\ge 0} \w b_{(i, 1^j)} \fa_{-i}(1_X)\fa_{-1}(x)^j q^{i+j}
     \right ) \vac, |1\rangle \right \rangle      \\
&=&\left \langle \prod_{1 \le i \le N, s \ge 1} \fa_s(x)^{m_{i,s}} \cdot
   \exp \left ( \sum_{i \ge 1, j\ge 0} \w b_{(i, 1^j)} \fa_{-i}(1_X)\fa_{-1}(x)^j q^{i+j}
     \right ) \vac, |1\rangle \right \rangle.
\end{eqnarray*}
Put $m_s = \sum_{i=1}^{N} m_{i,s}$ for every $s \ge 1$. Then,  \eqref{Fqxk.6} is equal to
\begin{eqnarray*}   
& &\left \langle \prod_{s \ge 1} \fa_s(x)^{m_s} \cdot
   \exp \left ( \sum_{i \ge 1, j\ge 0} \w b_{(i, 1^j)} \fa_{-i}(1_X)\fa_{-1}(x)^j q^{i+j}
     \right ) \vac, |1\rangle \right \rangle          \\
&=&\left \langle \prod_{s \ge 1} \fa_s(x)^{m_s} \cdot \prod_{i \ge 1, j\ge 0} \sum_t {1 \over t!} 
   \left (  \w b_{(i, 1^j)} \fa_{-i}(1_X)\fa_{-1}(x)^j q^{i+j}
     \right )^{t} \cdot \vac, |1\rangle \right \rangle             \\
&=&\prod_{s \ge 1} \left ( (-s)^{m_s} m_s! \sum_{t_0 + t_1 + \cdots + t_j + \cdots = m_s} 
     \prod_{j=0}^{+\infty} {(\w b_{(s, 1^j)} q^{s+j})^{t_j} \over  t_j!} \right ).
\end{eqnarray*}
Combining this with \eqref{Fqxk.5}, \eqref{Fqxk.2}, \eqref{Fqxk.6} and \eqref{Fqxk.7},
$F^{x,\ldots,x}_{k_1, \ldots, k_N}(q)$ is equal to
$$   
(q; q)_\infty^{-\chi(X)} \cdot 
(-1)^N  \sum_{\sum_{s \ge 1} (s+1) m_{i,s} = k_i + 2 \atop 1 \le i \le N}    
\prod_{i=1}^N {1 \over (\sum_{s \ge 1} s m_{i,s})!} \cdot  
$$
$$
\cdot \prod_{s \ge 1} \left ({(-s)^{m_s} m_s! \over \prod_i m_{i,s}!} 
      \sum_{t_0 + t_1 + \cdots + t_j + \cdots = m_s} 
     \prod_{j=0}^{+\infty} {(\w b_{(s, 1^j)} q^{s+j})^{t_j} \over  t_j!} \right )
$$
where $m_{i,s} \ge 0$ for every $i$ and $s$, and $m_s = \sum_{i=1}^{N} m_{i,s}$ 
for every $s \ge 1$.
\end{proof}

Our next result determines the universal constants $b_{(i, 1^j)}$ with $i \ge 2$ and $j \ge 0$.

\begin{theorem}   \label{bi1j}
Let the numbers $b_{(1^{j})}$ and $b_{(i, 1^j)}$
be from Lemma~\ref{Boi}. Let $\w b_{(i, 1^j)} = (j+1) b_{(1^{j+1})} = \sigma_1(j+1)$ 
if $i = 1$, and $\w b_{(i, 1^j)} = b_{(i, 1^j)}$ if $i > 1$.
\begin{enumerate}
\item[{\rm (i)}] If $i$ is an even positive integer, then $b_{(i, 1^j)} = 0$ for all $j \ge 0$. 

\item[{\rm (ii)}] Let $i > 1$ be odd. Then, 
$\displaystyle{{1 \over (i-1)!} \sum_{j \ge 0} b_{(i, 1^j)} q^{i+j}}$ is equal to
$$
\sum_{\sum_{1 \le s <i} (s+1) m_{s} = i+1 \atop 2 \nmid s} {1 \over (\sum_{2 \nmid s} s m_{s})!}   
   \prod_{2 \nmid s} \left (\sum_{\sum_{j \ge 0} t_j = m_s} 
   \prod_{j=0}^{+\infty} {\big ((-s) \w b_{(s, 1^j)} q^{s+j} \big )^{t_j} \over  t_j!} \right )  
$$
$$
- \sum_{a, s_1, \ldots, s_a, b, t_1, \ldots, t_b \ge 1 
   \atop {\sum_{u=1}^a s_u + \sum_{v = 1}^b t_v = i+1
       \atop {n_1 > \cdots > n_a, m_1 > \cdots > m_b
           \atop {\sum_{u=1}^a n_u s_u = \sum_{v=1}^b m_v t_v}}}} 
   \prod_{u= 1}^a {(-1)^{s_u} q^{n_u s_u} \over s_u! \cdot (1-q^{n_u})^{s_u}}
   \cdot \prod_{v= 1}^b {1 \over t_v! \cdot (1-q^{m_v})^{t_v}}.
$$
\end{enumerate}
\end{theorem}
\begin{proof}
(i) Setting $N = 1$ in Lemma~\ref{Fqxk}, we see that $F^{x}_{k}(q)$ is equal to
$$   
-(q; q)_\infty^{-\chi(X)} \cdot \sum_{\sum_{s \ge 1} (s+1) m_{s} = k + 2}   
{1 \over (\sum_{s \ge 1} s m_{s})!}   
\prod_{s \ge 1} \left (\sum_{\sum_{j \ge 0} t_j = m_s} 
\prod_{j=0}^{+\infty} {\big ((-s) \w b_{(s, 1^j)} q^{s+j} \big )^{t_j} \over  t_j!} \right ).  
$$
Comparing this with \eqref{ThmFkAlpha.0} which holds for all $k \ge 0$, we obtain
\begin{eqnarray*}
& &\sum_{\sum_{s \ge 1} (s+1) m_{s} = k + 2} {1 \over (\sum_s s m_{s})!}   
   \prod_{s \ge 1} \left (\sum_{\sum_{j \ge 0} t_j = m_s} 
   \prod_{j=0}^{+\infty} {\big ((-s) \w b_{(s, 1^j)} q^{s+j} \big )^{t_j} \over  t_j!} \right )  \\
&=&\sum_{a, s_1, \ldots, s_a, b, t_1, \ldots, t_b \ge 1 
   \atop {\sum_{u=1}^a s_u + \sum_{v = 1}^b t_v = k+2
       \atop {n_1 > \cdots > n_a, m_1 > \cdots > m_b
           \atop {\sum_{u=1}^a n_u s_u = \sum_{v=1}^b m_v t_v}}}}
   \prod_{u= 1}^a {(-1)^{s_u} q^{n_u s_u} \over s_u! \cdot (1-q^{n_u})^{s_u}}
   \cdot \prod_{v= 1}^b {1 \over t_v! \cdot (1-q^{m_v})^{t_v}}.
\end{eqnarray*}
The largest value of $s$ satisfying $\sum_{s \ge 1} (s+1) m_{s} = k + 2$ 
is given by $s = k+1$ together with $m_{k+1} = 1$. 
So the above identity can be rewritten as
\begin{eqnarray*}
& &{1 \over k!} \sum_{j \ge 0} 
    \w b_{(k+1, 1^j)} q^{(k+1)+j}    \\
&=&\sum_{\sum_{1 \le s < k+1} (s+1) m_{s} = k + 2} {1 \over (\sum_s s m_{s})!}   
   \prod_{s \ge 1} \left (\sum_{\sum_{j \ge 0} t_j = m_s} 
   \prod_{j=0}^{+\infty} {\big ((-s) \w b_{(s, 1^j)} q^{s+j} \big )^{t_j} \over  t_j!} \right )  \\
& &- \sum_{a, s_1, \ldots, s_a, b, t_1, \ldots, t_b \ge 1 
   \atop {\sum_{u=1}^a s_u + \sum_{v = 1}^b t_v = k+2
       \atop {n_1 > \cdots > n_a, m_1 > \cdots > m_b
           \atop {\sum_{u=1}^a n_u s_u = \sum_{v=1}^b m_v t_v}}}} 
   \prod_{u= 1}^a {(-1)^{s_u} q^{n_u s_u} \over s_u! \cdot (1-q^{n_u})^{s_u}}
   \cdot \prod_{v= 1}^b {1 \over t_v! \cdot (1-q^{m_v})^{t_v}}.
\end{eqnarray*}
Replacing $k+1$ by $i$, we conclude that $\displaystyle{{1 \over (i-1)!} \sum_{j \ge 0} 
\w b_{(i, 1^j)} q^{i+j}}$ is equal to
\begin{eqnarray}     \label{bi1j.1}
\sum_{\sum_{1 \le s <i} (s+1) m_{s} = i+1} {1 \over (\sum_s s m_{s})!}   
   \prod_{s \ge 1} \left (\sum_{\sum_{j \ge 0} t_j = m_s} 
   \prod_{j=0}^{+\infty} {\big ((-s) \w b_{(s, 1^j)} q^{s+j} \big )^{t_j} \over  t_j!} \right ) 
\end{eqnarray}
\begin{eqnarray}     \label{bi1j.2}
- \sum_{a, s_1, \ldots, s_a, b, t_1, \ldots, t_b \ge 1 
   \atop {\sum_{u=1}^a s_u + \sum_{v = 1}^b t_v = i+1
       \atop {n_1 > \cdots > n_a, m_1 > \cdots > m_b
           \atop {\sum_{u=1}^a n_u s_u = \sum_{v=1}^b m_v t_v}}}} 
   \prod_{u= 1}^a {(-1)^{s_u} q^{n_u s_u} \over s_u! \cdot (1-q^{n_u})^{s_u}}
   \cdot \prod_{v= 1}^b {1 \over t_v! \cdot (1-q^{m_v})^{t_v}}. 
\end{eqnarray}
Note that \eqref{bi1j.2} is equal to $0$ if $2|i$. 
Letting $i = 2$, we get $\displaystyle{\sum_{j \ge 0} \w b_{(2, 1^j)} q^{2+j} = 0}$.
Therefore, $\w b_{(2, 1^j)} = 0$ for every $j \ge 0$. 
Hence we have $b_{(2, 1^j)} = 0$ for every $j \ge 0$. 

Next, let $i > 2$ and $2|i$. Assume inductively that $b_{(s, 1^j)} = 0$ 
for every $j \ge 0$ whenever $2 \le s < i$ and $2|s$. 
Since \eqref{bi1j.2} is $0$,
$\displaystyle{{1 \over (i-1)!} \sum_{j \ge 0} 
b_{(i, 1^j)} q^{i+j}}$ is equal to 
$$
\sum_{\sum_{1 \le s <i} (s+1) m_{s} = i+1} {1 \over (\sum_s s m_{s})!}   
   \prod_{s \ge 1} \left (\sum_{\sum_{j \ge 0} t_j = m_s} 
   \prod_{j=0}^{+\infty} {\big ((-s) \w b_{(s, 1^j)} q^{s+j} \big )^{t_j} \over  t_j!} \right )
$$
The condition $\sum_{1 \le s <i} (s+1) m_{s} = i+1$ implies that $m_s > 0$ for 
some even integer $s < i$.
Hence $\displaystyle{{1 \over (i-1)!} \sum_{j \ge 0} b_{(i, 1^j)} q^{i+j}} = 0$ by induction.
So $b_{(i, 1^j)} = 0$ for all $j \ge 0$.

(ii) Follows immediately from (i), \eqref{bi1j.1} and \eqref{bi1j.2}.
\end{proof}

Note that $b_{(2i)} = 0$, $i \ge 1$ has been proved in \cite{Boi, BN} (see Lemma~\ref{Boi}). 
Next, using the universal constants $f_{(2,1^j)}, g_{(2,1^j)}$ and $h_{(2,1^j)}$,
we compute the generating series $F^{\alpha}_{1}(q)$ for a cohomology class $\alpha$
with $|\alpha| < 4$.

\begin{lemma}  \label{Fq1Alpha}
Let $f_{(2,1^j)}, g_{(2,1^j)}$ and $h_{(2,1^j)}$ be from Lemma~\ref{Boi},
and let $\alpha \in H^*(X)$ be a homogeneous class with $0 < |\alpha| < 4$. Then,
\begin{enumerate}
\item[{\rm (i)}] $\displaystyle{F^{1_X}_{1}(q) = \sum_{j \ge 0} \w f_{(2,1^j)} q^{2+j}}$
where $\w f_{(2,1^j)} = \chi(X) \cdot f_{(2,1^j)} + \langle K_X, K_X \rangle \cdot h_{(2,1^j)}$;

\item[{\rm (ii)}] $\displaystyle{F^{\alpha}_{1}(q) = (q; q)_\infty^{-\chi(X)} \cdot 
\sum_{j \ge 0} g_{(2,1^j)} q^{2+j} \cdot \langle \alpha, K_X \rangle}$.
\end{enumerate}
\end{lemma}
\begin{proof}
(i) Let $\alpha \in H^*(X)$ be an arbitrary cohomology class. 
Note that for all $n \ge 1$ and $A \in \fock$, we have 
$\displaystyle{\left \langle (n-1) (\mathfrak a_{-n} 
\mathfrak a_n)(K_X \alpha)A, |1\rangle \right \rangle = 0}$.
By \eqref{F-generating.1} and \eqref{char_thK=1.0}, 
\begin{eqnarray}     \label{Fq1Alpha.1}
   F^{\alpha}_{1}(q)
&=&\left \langle \fG_{1}(\alpha) \sum_n c\big (T_\Xn \big )q^n, |1\rangle \right \rangle 
                \nonumber    \\
&=&-\sum_{\ell(\lambda) = 3, |\lambda|=0} {1 \over \lambda!} 
   \left \langle \mathfrak a_{\lambda}(\alpha) \sum_n c\big (T_\Xn \big )q^n, 
   |1\rangle \right \rangle          \nonumber    \\
&=&-{1 \over 2} \left \langle (\mathfrak a_{-1}\mathfrak a_{-1}\mathfrak a_{2})(\alpha) 
   \sum_n c\big (T_\Xn \big )q^n, |1\rangle \right \rangle   \nonumber    \\
&=&-{1 \over 2} \left \langle \mathfrak a_{2}(\alpha) 
   \sum_n c\big (T_\Xn \big )q^n, |1\rangle \right \rangle.
\end{eqnarray}

Set $\alpha = 1_X$. Put $\w f_\mu = \chi(X) \cdot f_\mu + \langle K_X, K_X \rangle \cdot h_\mu$.
By Lemma~\ref{Boi}, $F^{1_X}_{1}(q)$ equals
\begin{eqnarray*}        
& &-{1 \over 2} \left \langle \mathfrak a_{2}(1_X) 
   \exp \left ( \sum_{\mu \in \cp} q^{|\mu|} \w f_\mu \fa_{-\mu}(x) \right ) \vac, 
   |1\rangle \right \rangle  \\
&=&-{1 \over 2} \left \langle \mathfrak a_{2}(1_X) 
   \exp \left ( \sum_{j \ge 0} q^{2+j} \w f_{(2,1^j)} \fa_{-2}(x)\fa_{-1}(x)^j \right ) \vac, 
   |1\rangle \right \rangle   \\
&=&-{1 \over 2} \left \langle \mathfrak a_{2}(1_X) 
   \left ( \sum_{j \ge 0} q^{2+j} \w f_{(2,1^j)} \fa_{-2}(x)\fa_{-1}(x)^j \right ) \vac, 
   |1\rangle \right \rangle   \\
&=&\sum_{j \ge 0} \w f_{(2,1^j)} q^{2+j}.
\end{eqnarray*}

(ii) Let $0 < |\alpha| < 4$. Again by \eqref{Fq1Alpha.1} and Lemma~\ref{Boi}, 
$F^{\alpha}_{1}(q)$ is equal to
\begin{eqnarray*}        
& &-{1 \over 2} (q; q)_\infty^{-\chi(X)} \left \langle \mathfrak a_{2}(\alpha) 
   \exp \left ( \sum_{\mu \in \cp} q^{|\mu|}  
   g_\mu \fa_{-\mu}(K_X)   \right ) \vac, |1\rangle \right \rangle  \\
&=&-{1 \over 2} (q; q)_\infty^{-\chi(X)} \left \langle \mathfrak a_{2}(\alpha) 
   \exp \left ( \sum_{j \ge 0} q^{2+j}  g_{(2,1^j)} (\fa_{-2}\fa_{-1}^j)(K_X)   
   \right ) \vac, |1\rangle \right \rangle   \\
&=&-{1 \over 2} (q; q)_\infty^{-\chi(X)} \left \langle \mathfrak a_{2}(\alpha) 
   \left ( \sum_{j \ge 0} q^{2+j}  g_{(2,1^j)} \fa_{-2}(K_X)\fa_{-1}(x)^j  
   \right ) \vac, |1\rangle \right \rangle.
\end{eqnarray*}
Therefore, $\displaystyle{F^{\alpha}_{1}(q) = (q; q)_\infty^{-\chi(X)} \cdot 
\sum_{j \ge 0} g_{(2,1^j)} q^{2+j} \cdot \langle \alpha, K_X \rangle}$
when $0 < |\alpha| < 4$.
\end{proof}

\begin{proposition}   \label{CorPropch1Alpha}
Let the numbers $g_{(2,1^j)}$ and $h_{(2,1^j)}$ be from Lemma~\ref{Boi}. Then,
$g_{(2,1^j)} = -h_{(2,1^j)}$. Moreover, $\sum_{j \ge 0} g_{(2,1^j)} q^{2+j}$ is 
the coefficient of $z^0$ in
$$
{1 \over 2} \left ( \sum_{n} {(n-1)q^n \over (1 - q^n)^2} 
+ \sum_n {(qz)^{n} \over 1-q^n} 
\cdot \left (\sum_m {z^{-2m} \over (1-q^m)^2} + 2 \sum_{m_1 > m_2} 
{z^{-m_1} \over 1-q^{m_1}} {z^{-m_2} \over 1-q^{m_2}} \right ) \right ).
$$
\end{proposition}
\begin{proof}
For simplicity, denote the previous line by $A(z)$.
Let $X$ be a smooth projective surface with $\chi(X) = 0$ and 
$\langle K_X, K_X \rangle \ne 0$. 
On one hand, applying Lemma~\ref{Fq1Alpha}~(i) and 
Proposition~\ref{Propch1Alpha} to $F^{1_X}_{1}(q)$, we conclude that 
$\sum_{j \ge 0} h_{(2,1^j)} q^{2+j}$ is the coefficient of $z^0$ in $-A(z)$.
On the other hand, applying Lemma~\ref{Fq1Alpha}~(ii) and 
Proposition~\ref{Propch1Alpha} to $F^{K_X}_{1}(q)$, we see that 
$\sum_{j \ge 0} g_{(2,1^j)} q^{2+j}$ is the coefficient of $z^0$ in $A(z)$.
It follows that $g_{(2,1^j)} = -h_{(2,1^j)}$ for every $j \ge 0$.
\end{proof}

\begin{remark}   \label{RmkZero}
Let $N \ge 1$.
Let $\alpha_1, \ldots, \alpha_N \in H^*(X)$ be homogeneous classes such that 
$K_X \alpha_i = e_X \alpha_i = 0$ for all $1 \le i \le N$,
and let $k_1, \ldots, k_N \ge 0$.
\begin{enumerate}
\item[{\rm (i)}]  
As in the proof of Lemma~\ref{Fqxk}, we have
\begin{eqnarray*} 
F^{\alpha_1, \ldots, \alpha_N}_{k_1, \ldots, k_N}(q)       
= (q; q)_\infty^{-\chi(X)} \left \langle 
   \left ( \prod_{i=1}^N \fG_{k_i}(\alpha_i) \right )
   \exp \left ( \sum_{\mu \in \cp}  b_\mu \fa_{-\mu}(1_X) q^{|\mu|} \right ) \vac, 
   |1\rangle \right \rangle. 
\end{eqnarray*}
In principle, together with Theorem~\ref{ThmFk1NAlpha1N},
this allows us to determine many of the universal constants $b_\mu$ in Lemma~\ref{Boi}.

\item[{\rm (ii)}] 
In particular, $F^{\alpha_1}_{k_1}(q) = 0$ if $|\alpha_1| < 4$.
This matches with Proposition~\ref{ThmFkAlpha}~(i).
\end{enumerate}
\end{remark}

\end{document}